\documentclass[aos,MSNbibl,seceqn,citesort,dvips]{arximspdf}

\doi{10.1214/10-AOS869}
\volume{39}
\issue{3}
\pubyear{2011}
\firstpage{1551}
\lastpage{1579}

\makeatletter

\newtheorem{cor}{Corollary}[section]
\newtheorem{lem}{Lemma}[section]
\newtheorem{thr}{Theorem}[section]

\newproclaim{rem}{Remark}[section]
\newproclaim{Assumption}{Assumption}
\newproclaim{definition}{Definition}

\makeatother

\begin{document}
\begin{frontmatter}

\title{Asymptotic Bayes-optimality under sparsity of some multiple testing procedures}
\runtitle{ABOS of multiple testing procedures}

\begin{aug}
\author[A,D]{\fnms{Ma{\l}gorzata}~\snm{Bogdan}\corref{}\thanksref{t1}\ead[label=e1]{Malgorzata.Bogdan@pwr.wroc.pl}\ead[label=e6]{bogdanm@stat.purdue.edu}},
\author[B]{\fnms{Arijit}~\snm{Chakrabarti}\ead[label=e3]{arc@isical.ac.in}},
\author[C]{\fnms{Florian}~\snm{Frommlet}\thanksref{t2}\ead[label=e4]{Florian.Frommlet@univie.ac.at}}
and
\author[D,B]{\fnms{Jayanta~K.}~\snm{Ghosh}\ead[label=e2]{ghosh@purdue.edu}\ead[label=e5]{jayanta@isical.ac.in}}
\runauthor{Bogdan, Chakrabarti, Frommlet and Ghosh}
\affiliation{Wroc{\l}aw University of Technology and Purdue University,
Indian Statistical Institute, University of Vienna, and Purdue
University and~Indian~Statistical~Institute}
\address[A]{M. Bogdan\\
Institute of Mathematics\\
\quad and Computer Science\\
Wroc{\l}aw University of Technology\\
Ul. Janiszewskiego 14a\\
50-370 Wroclaw\\
Poland\\
\printead{e1}}
\address[B]{A. Chakrabarti\\
J. K. Ghosh\\
Indian Statistical Institute\\
203 B.T. Road\\
Kolkata 700108, West Bengal\hspace*{4pt}\\
India\\
\printead{e3}\\
\hphantom{E-mail: }\printead*{e5}}
\address[C]{F. Frommlet\\
Department of Statistics\\
\quad and Decision Support Systems\\
University of Vienna \\
Br\"{u}nnerstra{\ss}e 72\\
1210 Vienna\\
Austria\\
\printead{e4}}
\address[D]{J. K. Ghosh\\
M. Bogdan\\
Department of Statistics\\
Purdue University \\
150 North University Street\\
West Lafayette, Indiana 47907\\
USA\\
\printead{e2}\\
\hphantom{E-mail: }\printead*{e6}}
\end{aug}

\thankstext{t1}{Supported in part by Grants 1 P03A 01430 and NN 201414139 of the Polish
Ministry of Science and Higher Education.}

\thankstext{t2}{Supported in part by the WWTF Grant MA09-007.}

\received{\smonth{2} \syear{2010}}
\revised{\smonth{12} \syear{2010}}

%
\begin{abstract}
Within a Bayesian decision theoretic framework we investigate some
asymptotic optimality properties of a large class of multiple testing
rules. A parametric setup is considered, in which observations come
from a normal scale mixture model and the total loss is assumed to be
the sum of losses for individual tests. Our model can be used for
testing point null hypotheses, as well as to distinguish large signals
from a multitude of very small effects. A~rule is defined to be
asymptotically Bayes optimal under sparsity (ABOS), if within our
chosen asymptotic framework the ratio of its Bayes risk and that of the
Bayes oracle (a rule which minimizes the Bayes risk) converges to one.
Our main interest is in the asymptotic scheme where the proportion $p$
of ``true'' alternatives converges to zero.

We fully characterize the class of fixed threshold multiple testing
rules which are ABOS, and hence derive conditions for the asymptotic
optimality of rules controlling the Bayesian False Discovery Rate
(BFDR). We finally provide conditions under which the popular
Benjamini--Hochberg (BH) and Bonferroni procedures are ABOS and show
that for a wide class of sparsity levels, the threshold of the former
can be approximated by a nonrandom threshold.

It turns out that while the choice of asymptotically optimal FDR levels
for BH depends on the relative cost of a type I error, it is almost
independent of the level of sparsity. Specifically, we show that when
the number of tests $m$ increases to infinity, then BH with FDR level
chosen in accordance with the assumed loss function is ABOS in the
entire range of sparsity parameters $p\propto m^{-\beta}$, with $\beta
\in(0,1]$.
\end{abstract}

%
\begin{keyword}[class=AMS]
\kwd[Primary ]{62C25}
\kwd{62F05}
\kwd[; secondary ]{62C10}.
\end{keyword}
\begin{keyword}
\kwd{Multiple testing}
\kwd{FDR}
\kwd{Bayes oracle}
\kwd{asymptotic optimality}.
\end{keyword}

\end{frontmatter}

\section{Introduction}

Multiple testing has emerged as a very important problem in statistical
inference because of its applicability in understanding large data sets
involving many parameters. A prominent area of the application of
multiple testing is microarray data analysis, where one wants to
simultaneously test expression levels of thousands of genes (see, e.g.,
\cite{ETST,ET,S3,GW2,MPRR,Sk3,SB} or~\cite{STS}). Various ways of performing multiple
tests have been proposed in the literature over the years, typically
differing in their objective. Among the most popular classical multiple
testing procedures, one finds the Bonferroni correction, aimed at
controlling the family wise error rate (FWER) and the
Benjamini--Hochberg procedure~\cite{BH}, which controls the false
discovery rate (FDR). A wide range of empirical Bayes (e.g., see
\cite{E,ET,ETST,BGT} and~\cite{TS}) and full Bayes
tests (see, e.g.,~\cite{MPRR,DMT,BGT} and~\cite{SB})
have also been proposed and are used extensively in such problems.

In the classical setting, a multiple testing procedure is considered to
be \textit{optimal} if it maximizes the number of true discoveries, while
keeping one of the type~I error measures (like FWER, FDR or the
expected number of false positives) at a certain, fixed level. In this
context, it is shown in~\cite{GR} that the Benjamini--Hochberg
procedure (henceforth called BH) is optimal within a large class of
step-up multiple testing procedures controlling FDR.
In recent years many new multiple testing procedures, which have some
optimality properties in the classical sense, have been proposed (e.g.,
\cite{Lehmann,Chi,Pena} or~\cite{Roquain}). In
\cite{Finner} an asymptotic analysis is performed and new step-up and
step-up-down procedures, which maximize the asymptotic power while
controlling the asymptotic FDR, are introduced. Also, in~\cite{S3} and
\cite{SCai} two classical oracle procedures for multiple testing are
defined. The oracle procedure proposed in~\cite{S3} maximizes the
expected number of true positives where the expected number of false
positives is kept fixed. This procedure requires the knowledge of the
true distribution for all test statistics and is rather difficult to
estimate without further assumptions on the process generating the
data. The oracle proposed in~\cite{SCai} assumes that the data is
generated according to a two-component mixture model. It aims at
maximizing the marginal false nondiscovery rate (mFNR), while
controlling the marginal false discovery rate (mFDR) at a given level.
In~\cite{SCai} a data-driven adaptive procedure is developed, which
asymptotically attains the performance of the oracle procedure for any
fixed (though unknown) proportion $p$ of alternative hypothesis.

In this paper we take a different point of view and analyze the
properties of multiple testing rules from the perspective of Bayesian
decision theory. We assume for each test fixed losses $\delta_0$ and
$\delta_A$ for type I and type II errors, respectively, and define the
overall loss of a multiple testing rule as the sum of losses incurred
in each individual test. We feel that such an approach is natural in
the context of testing, where the main goal is to detect significant
signals, rather than estimate their magnitude. In the specific case
where $\delta_0=\delta_A=1$, the total loss is equal to the number of
misclassified hypotheses. Also, we consider the asymptotic scheme,
under which the proportion $p$ of ``true'' alternatives among all tests
converges to zero as the number of tests $m$ goes to infinity, and
restrict our attention to the signals on the \textit{verge of
detectability}, which can be asymptotically detected with the power in $(0,1)$.

In recent years, substantial efforts have been made to understand the
properties of multiple testing procedures under sparsity, that is, in
the case where $p$ is very small (e.g.,~\cite{Djin,Djin2,MR,JinCai,CJ}).
A major theoretical breakthrough was
made in~\cite{A}, where
it has been shown that the Benjamini--Hochberg procedure can be used
for estimating a sparse vector of means, while the level of sparsity
can vary considerably. In~\cite{A} independent normal observations
$X_i, i=1,\ldots,m$, with unknown means $\mu_i$ and known variance are
considered. Among the studied parameter spaces are the $l_0[p_m]$
balls, which consist of those real $m$-vectors for which the fraction
of nonzero elements is at most $p_m$. A data-adaptive thresholding
estimator for the unknown vector of means is proposed using the
Benjamini--Hochberg rule at the FDR level $\alpha_m \geq\frac{\gamma
}{\log m}$ for some $\gamma> 0$ and all $m>1$. If the FDR control
level $\alpha_m$ converges to $\alpha_0 \in[0,1/2]$, this estimator is
shown to be asymptotically minimax, simultaneously for a large class of
loss functions (and in fact for many different types of sparsity
classes including the $l_0$ balls),\vspace*{1pt} as long as $p_m$ is in the range
$[\frac{\log^5 m}{m}, m^{-\xi}]$, with $\xi\in(0,1)$.


In this paper we provide new theoretical results, which illustrate the
asymptotic optimality properties of BH under sparsity in the context of
Bayesian decision theory. BH is a very interesting procedure to analyze
from this point of view, since, despite its frequentist origin, it
shares some of the major strengths of Bayesian methods. Specifically,
as shown in~\cite{ET} and~\cite{GW1}, BH can be understood as an
empirical Bayes approximation to the procedure controlling the
``Bayesian'' False Discovery Rate (BFDR). This approximation relies
mainly on estimation of the distribution generating the data by the
empirical distribution function. In this way, similarly to standard
Bayes methods, it gains strength by combining information from all the
tests. The major issue addressed in this paper is the relationship
between BFDR control and optimization of the Bayes risk. Our research
was motivated mainly by the good properties of BH with respect to the
misclassification rate under sparsity, documented in~\cite{BGT,BGOT} and
\cite{GW1}. The present paper lends theoretical support to
these experimental findings, by specifying a large range of loss
functions for which BH is asymptotically optimal in a Bayesian decision
theoretic context.

The outline of the paper is as follows. In Section~\ref{Section:SM} we
define and discuss our model, and we introduce the decision theoretic
and asymptotic framework of the paper. The Bayes oracle, which
minimizes the Bayes risk, is presented, which applies a \textit{fixed
threshold} critical region for each individual test. Conditions are
formulated under which the asymptotic power of this rule is larger than
0, but smaller than 1. Two different levels of sparsity, the extremely
sparse case and a slightly denser case, are defined, which play a
prominent role throughout the paper.

In Section~\ref{Section:ABOS} we compute the asymptotic risk of the
Bayes oracle, and we formally define the concept of asymptotic Bayes
optimality under sparsity (ABOS).
We then study fixed threshold tests in great detail and fully
characterize the class of fixed threshold testing rules being ABOS.
In the subsequent Section~\ref{Section:BFDR}
we study fixed threshold multiple testing rules which make use of the
unknown model parameters to control the Bayesian False Discovery Rate
(BFDR) exactly at a given level $\alpha$. We provide conditions for
such rules to be ABOS and also consider ABOS of the closely related
fixed threshold tests using the asymptotic approximation of the BH
threshold $c_{\mathrm{GW}}$, introduced by Genovese and Wasserman~\cite{GW1}.
Specifically, in Corollary~\ref{Lu1} we show that if $p\propto
m^{-\beta
}$ for some $\beta>0$, then the asymptotically optimal BFDR levels
depend mainly on the ratio of loss functions for type~I and type II
errors and are independent of $\beta$.

The main results of the paper are included in Section \ref
{Section:rules}, where we specify conditions under which the Bonferroni
rule as well as the Benjamini--Hochberg procedure are ABOS.
Specifically, Theorem~\ref{BHthreshold} shows that when FDR levels $\alpha_m
\rightarrow\alpha_{\infty}<1$ satisfy the conditions of optimality of
BFDR controlling rules, then the difference between the random
threshold of BH and the Genovese--Wasserman threshold $c_{\mathrm{GW}}$
converges to 0 for any sequence of sparsity parameters $p_m \propto
m^{-\beta}$, with $\beta\in(0,1)$. Theorem~\ref{BHthr} shows that for the same
FDR levels BH is ABOS whenever $p_m \propto m^{-\beta}$, with $\beta
\in
(0,1]$. Thus, our results show that BH adapts to the unknown level of
sparsity. However, we also show that the optimal FDR controlling level
depends on the relative cost of a type I error---it should be chosen to
be small if the relative cost of the type I error is large.
Specifically, within our asymptotic framework, the Benjamini--Hochberg
rule controlling the FDR at a fixed level $\alpha\in(0,1)$ is ABOS
for a wide range of sparsity levels, provided that the ratio of losses
for type I and type II errors converges to zero at a slow rate which
can vary widely. When the loss ratio is constant, similar optimality
results hold if the FDR controlling level slowly converges to zero.

Section~\ref{Disc} contains a discussion and directions for further
research. The proof of the asymptotic optimality of BH can be found in
Section~\ref{App}, while the remaining lengthy proofs can be found in
the supplemental report~\cite{App}.

\section{Statistical model and asymptotic framework} \label{Section:SM}

Suppose we have $m$ independent observations $X_1,\ldots,X_m$, and
assume that each $X_i$ has a normal $N(\mu_i, \sigma_{\varepsilon}^2)$
distribution. Here $\mu_i$ represents the effect under investigation,
and $\sigma_{\varepsilon}^2$ is the variance of the random noise (e.g.,
the measurement error). We assume that each $\mu_i$ is an independent
random variable, with distribution determined by the value of the
unobservable random variable $\nu_i$, which takes values 0 and 1 with
probabilities $1-p$ and $p$, respectively, for some $p \in(0,1)$. We
denote by $H_{0i}$ the event that $\nu_i = 0$, while $H_{Ai}$ denotes
the event $\nu_i = 1$. We will refer to these events as the null and
alternative hypotheses. Under $H_{0i}$, $\mu_i$ is assumed to have a
$N(0,\sigma_0^2)$ distribution (where $\sigma_0^2 \geq0$), while under
$H_{Ai}$ it is assumed to have a $N (0,\sigma_0^2+\tau^2)$ distribution
(where $\tau^2 > 0$). Hence, we are really modeling the $\mu_i$'s as
i.i.d. r.v.'s from the following mixture distribution:
%
\begin{equation}\label{modelmi}
\mu_i \sim(1-p) N(0,\sigma_0^2)+pN(0,\sigma_0^2+\tau^2) .
\end{equation}

This implies that the marginal
distribution of $X_i$ is the scale mixture of normals, namely,
%
\begin{equation}\label{modelX}
X_i \sim(1-p) N(0,\sigma^2)+pN(0,\sigma^2+\tau^2) ,
\end{equation}
where $\sigma^2=\sigma_{\varepsilon}^2+\sigma_0^2$.

Note that in the case where $\sigma_0^2=0$, $H_{0i}$ corresponds to the
point null hypothesis that $\mu_i=0$, and $H_{Ai}$ says that $\mu_i
\neq0$ [since under $H_{Ai}$ $P(\mu_i=0)=0)$]. Thus this model can be
used for simultaneously testing if the means of the $X_i$'s are zero or
not. Allowing $\sigma_0^2>0$ greatly extends the scope of the
applications of the proposed mixture model under sparsity. In many
multiple testing problems it seems unrealistic to assume that the vast
majority of effects are exactly equal to zero. For example, in the
context of locating genes influencing quantitative traits, it is
typically assumed that a trait is influenced by many genes with very
small effects, so called polygenes. Such genes form a background, which
can be modeled by the null component of the mixture. In this case the
main purpose of statistical inference is the identification of a small
number of significant ``outliers,'' whose impact on the trait is
substantially larger than that of the polygenes. These important
``outlying'' genes are modeled by the nonnull component of the mixture.

In the remaining part of the paper we will assume that the variance of
$X_i$ under the null hypothesis, $\sigma^2$, is known. This assumption
is often used in the literature on the asymptotic properties of
multiple testing procedures (see, e.g.,~\cite{A} or~\cite{Djin}).
Some discussion concerning the general issue of estimation of
parameters in sparse mixtures is provided in Section~\ref{Disc}.
\begin{rem}
Note that given $\mu_i$, the distribution of $X_i$ is a
location shift of the distribution under the null. This is the setting
in which multiple testing is typically analyzed in the classical
context. In our extended Bayesian model, the choice of a normal $
N(0,\sigma_0^2 + \tau^2)$ prior for $\mu_i$ under the alternative
results in a corresponding normal $ N(0,\sigma^2+ \tau^2)$ marginal
distribution for $X_i$, which differs from the null distribution only
by a larger scale parameter. The proposed mixture model for $X_i$ is a
specific example of the two-groups model, which was discussed in a
wider nonparametric context, for example, in~\cite{ETST,E,BGT} and
\cite{GW2}. Similar Gaussian mixture models for multiple
testing were considered, for example, in~\cite{CJ} and~\cite{E1}.
Restricting attention to scale mixtures of normal distributions allows
us to reduce the technical complexity of the proofs and to concentrate
on the main aspects of the problem. Moreover, we believe that the
proposed model is applicable in many practical situations,
when there are no prior expectations concerning the sign of $\mu_i$.
Our asymptotic results may be extended to the situation when the
distribution of $\mu_i$ under the alternative is not symmetric about 0.
Namely, the techniques presented in the related report~\cite{FBC} can
be used for a similar asymptotic analysis when the ``alternative''
normal distribution $N(0,\sigma_0^2+\tau^2)$ of $\mu_i$ in the model
(\ref{modelmi}) is replaced by a general scale distribution, with the
scale parameter playing role of~$\tau$. A manuscript dealing with this
case is in preparation.
\end{rem}

We consider a Bayesian decision theoretic formulation of the multiple
testing problem of testing $H_{0i}$ versus $H_{Ai}$, for $i=1,\ldots,m$
simultaneously. For each $i$, there are two possible ``states of
nature,'' namely $H_{0i}$ with $X_i \sim N(0, \sigma^2)$ or $H_{Ai}$
with $X_i \sim N(0, \sigma^2+\tau^2)$, that occur with probabilities
$(1-p)$ and $p$, respectively.
Table~\ref{table1} defines the matrix of losses for making a decision in the $i$th test.

\begin{table}
\tablewidth=200pt
\caption{Matrix of losses}\label{table1}
\begin{tabular*}{\tablewidth}{@{\extracolsep{\fill}}lcc@{}}
\hline
& \textbf{Choose} $\bolds{H_{0i}}$ & \textbf{Choose} $\bolds{H_{Ai}}$\\
\hline
$H_{0i}$ true & 0 & $\delta_0$\\
$H_{Ai}$ true & $\delta_A$ & 0\\
\hline
\end{tabular*}
\end{table}

We assume that the overall loss in the multiple testing procedure is
the sum of losses for individual tests. Thus our approach is based on
the notion of an additive loss function, which goes back to~\cite{Leh1}
and~\cite{Leh2}, and seems to be implicit in most of the current formulations.

Under an additive loss function, the compound Bayes decision problem
can be solved as follows. It is easy to see that the expected value of
the total loss is minimized by a procedure which simply applies the
Bayesian classifier to each individual test. For each $i$, this leads
to choosing the alternative hypothesis $H_{Ai}$ in cases such that
%
\begin{equation}\label{BOdef}
\frac{\phi_A(X_i)}{\phi_0(X_i)}\geq\frac{(1-p) \delta_0}{p \delta
_A} ,
\end{equation}
where $\phi_A$ and $\phi_0$ are the densities of $X_i$ under the
alternative and null hypotheses, respectively.

After substituting in the formulas for the appropriate normal
densities, we obtain the optimal rule
%
\begin{equation}\label{BO}
\mbox{Reject } H_{0i} \qquad\mbox{if } \frac{X_i^2}{\sigma^2}\geq c^2 ,
\end{equation}
where
%
\begin{equation}\label{crit}
c^2=c^2_{\tau,\sigma,f,\delta}=\frac{\sigma^2+\tau^2}{\tau
^2}
\biggl(\log\biggl(\biggl(\frac{\tau}{\sigma}\biggr)^2+1\biggr)+2\log
(f\delta
)\biggr)
\end{equation}
with $f=\frac{1-p}{p}$
and $\delta=\frac{\delta_0}{\delta_A}$. We call this rule a \textit{Bayes
oracle}, since it makes use of the unknown parameters of the mixture,
$\tau$ and $p$,
and therefore is not attainable in finite samples.

Using standard notation from the theory of testing, we define the
probability of a type I error as
\[
t_{1i}=P_{H_{0i}}(H_{0i} \mbox{ is rejected})
\]
and the probability of a type II error as
\[
t_{2i}=P_{H_{Ai}}(H_{0i} \mbox{ is accepted}) .
\]

Note that under our mixture model the marginal distributions of $X_{i}$
under the null and alternative hypotheses do not depend on $i$, and the
threshold of the Bayes oracle is also the same for each test. Hence,
when calculating the probabilities of type I errors and type II errors
for the Bayes oracle, we can, and will henceforth, suppress $i$ from
$t_{1i}$ and $t_{2i}$. The same remark also applies to any fixed
threshold procedure which, for each $i$, rejects $H_{0i}$ if
$X_i^2/{\sigma^2} > K$ for some constant $K$.

\subsection{The asymptotic framework}\label{assA}

We now want to motivate the asymptotic framework which will be
formally introduced below as Assumption~\ref{Assumption(A)}.
Let $\gamma=(p, \tau^2, \sigma^2, \delta_0, \delta_A)$ be the vector
of parameters defining the Bayes oracle
(\ref{crit}). In our asymptotic analysis, we will consider infinite
sequences of such $\gamma$'s. A natural example of such a situation
arises when the number of tests $m$ increases to infinity, and the
vector $\gamma$ varies with the number of tests $m$. But here we are
actually trying to understand, in a unified manner, the general
limiting problem when $\gamma$ varies through a sequence.

The threshold (\ref{crit}) depends on $\tau$ and $\sigma$ only through
$u=(\frac{\tau}{\sigma})^2$. Note that $u$ is a natural
scale for measuring the strength of the signal in terms of the variance
of $X_i$ under the null. We also introduce another parameter
$v=uf^2\delta^2$, which can be used to simplify the formula for the
optimal threshold
%
\begin{equation}\label{crit_uv}
c^2_{u,v}= \biggl(1+\frac{1}{u}\biggr)\bigl(\log v+\log(1+1/u)\bigr) .
\end{equation}

Observe that under the alternative $\frac{X_i}{\sigma}$ has a normal
$N(0,1+u)$ distribution. Thus the probability of a type II error using
the Bayes oracle is given by
%
\begin{equation}\label{type2,1}
t_2=P\biggl(Z^2<\frac{1}{u+1} c^2_{u,v}\biggr) ,
\end{equation}
where $Z$ is a standard normal variable.

From\vspace*{-1pt} (\ref{type2,1}) it follows that given an arbitrary infinite
sequence of $\gamma$'s, the limiting power of the Bayes oracle\vspace*{1pt} is
nonzero only if the corresponding sequence $\frac{c^2_{u,v}}{u+1}$
remains bounded. We will restrict ourselves to such sequences, since
otherwise even the Bayes oracle cannot guarantee nontrivial inference
in the limit and all rules will perform poorly.

The focus of this paper is the study of the inference problem when $p
\rightarrow0$, and the goal is to find procedures which will
efficiently identify signals under such circumstances. To clarify these
ideas, consider the situation where $p \rightarrow0$ and $\log(\delta
)=o(\log p)$. It is immediately evident from (\ref{crit}) that in this
situation $c^2=c^2_{u,v}$ diverges to infinity. Hence\vspace*{-1pt} $\frac
{c^2_{u,v}}{u+1}$ remains bounded only when the signal magnitude $u$
diverges to infinity, in which case $\frac{c^2_{u,v}}{u+1}\propto
\frac
{\log v}{u}$. This explains two of the three asymptotic conditions we
impose in Assumption~\ref{Assumption(A)}. The third condition $v
\rightarrow\infty$
pragmatically ensures that $\delta$ is not allowed to converge to zero
too quickly.
\renewcommand{\theAssumption}{(A)}
\begin{Assumption}\label{Assumption(A)}
A sequence of vectors $\{\gamma_t = (p_t, \tau
_t^2, \sigma_{t}^2, \delta_{0t},\delta_{At}); t \in\{1,2,\ldots\} \}
$ satisfies this assumption if the corresponding sequence of parameter
vectors, $\theta_t=(p_t, u_t, v_t)$, fulfills the following conditions:
$p_t\rightarrow0$, $u_t \rightarrow\infty$, $v_t\rightarrow\infty$ and
$\frac{\log v_t}{u_t}\rightarrow C \in(0,\infty)$, as $t \rightarrow
\infty$.
\end{Assumption}
\begin{rem}
We do not allow $C=\infty$ in Assumption~\ref{Assumption(A)} because
then the limit of
the probability of a type II error for Bayes oracle is equal to one,
and signals cannot be identified.
If $C=0$, then the oracle has a limiting power equal to one. Such a
situation can occur naturally if the number of replicates used to
calculate $X_i$ increases to infinity as $p\rightarrow0$ (see, e.g.,
\cite{FBC}). However, in this article we will restrict ourselves to $C
\in(0, \infty)$, that is, the case where the asymptotic power is
smaller than one. The corresponding parametric region might be thought
of as being at ``the verge of detectability.'' The extension of the
asymptotic results presented in this paper to the case when $C=0$ as
well as to some cases when $p$ does not converge to zero can be found
in~\cite{BCFG}, which is an extended version of this manuscript.
Specifically, Theorems~\ref{riskopt},~\ref{main} and~\ref{Lu0} below
hold in exactly the same form even when the condition $p\rightarrow0$
is eliminated from Assumption~\ref{Assumption(A)}.
\end{rem}
\begin{rem} \label{Rem_gen}
We will frequently consider the generic situation
%
\begin{equation}\label{bound_logdelta}
\log\delta=o(\log p) .\vadjust{\goodbreak}
\end{equation}
In that case Assumption~\ref{Assumption(A)} reduces to $p\rightarrow0$,
$u \rightarrow
\infty$, $v\rightarrow\infty$ and
$-\frac{2 \log p}{u}\rightarrow C \in(0,\infty)$ and specifies the
relationship between the magnitude $u$ of asymptotically detectable
signals and the sparsity parameter $p$. Interestingly, the relationship
$u\propto- \log p$, can be related to asymptotically least-favorable
configurations for $l_0[p]$ balls discussed in Section 3.1 of~\cite{A}.
Ignoring constants, the typical magnitudes of observations
corresponding to such signals will be similar to the threshold of the
minimax hard thresholding estimator corresponding to the parameter
space~$l_0[p]$.
\end{rem}

\textit{Notation}: We will usually suppress the index $t$ of the elements
of the vector $\gamma_t$ and $\theta_t$. Unless otherwise stated,
throughout the paper the notation $o_t$ will denote an infinite
sequence of terms indexed by $t$, which go to zero when $t\rightarrow
\infty$. In many cases $t$ is the same as the number of tests $m$, and
in such cases the notation $o_t$ will be replaced by $o_m$.\vspace*{8pt}

In case of $m \rightarrow\infty$ we will consider specifically two
different levels of sparsity. The first, the extremely sparse case, is
characterized by
%
\begin{equation}\label{sparse}
mp_m \rightarrow s \in(0, \infty] \quad\mbox{and}\quad \frac{\log
(mp_m)}{\log
m}\rightarrow0 .
\end{equation}
Condition (\ref{sparse}) is satisfied, for example, when $p_m\propto
\frac{1}{m}$. In this situation the expected number of ``signals'' does
not increase with $m$, which makes it impossible to consistently
estimate the mixture parameters.
The second, ``denser'' case is characterized by
%
\begin{equation}\label{dense}
p_m \rightarrow0 \quad\mbox{and}\quad \frac{\log(mp_m)}{\log m}\rightarrow
C_p\in
(0,1] ,
\end{equation}
which includes $p_m\propto m^{-\beta}$ for $0 < \beta< 1$.

\section{Asymptotic Bayes-optimality under sparsity} \label{Section:ABOS}

We start by computing type I and type II error rates of the Bayes oracle.
As usual $\Phi$ denotes the cumulative distribution function and $\phi$
the density of the standard normal distribution.
\begin{lem}\label{type1lem}
Under Assumption~\ref{Assumption(A)} the probabilities of type \textup{I} and
type~\textup{II} error
using the Bayes oracle are given by the following equations:
%
\begin{eqnarray}\label{t1,0}
t_{1}&=&e^{-C/2}\sqrt{\frac{2}{\pi v \log v}} (1+o_{t}),
\\
\label{t2n}
t_{2}&=&\bigl(2\Phi\bigl(\sqrt{C}\bigr)-1\bigr) (1+o_{t}) .
\end{eqnarray}
\end{lem}
\begin{pf}
Note that $t_{1}=P(|Z|>c_{u,v})$. Moreover,
%
\begin{equation}\label{th1}
c^2_{u,v}=(1+z_{u,v}) \log v ,
\end{equation}
where $\lim_{u \rightarrow\infty, v\rightarrow\infty} z_{u,v} u=1$.
Therefore, we obtain
\[
\phi(c_{u,v})\sqrt{2\pi v}= \exp\biggl(\frac{-z_{u,v}\log
v}{2}\biggr) ,
\]
which, together with Assumption~\ref{Assumption(A)}, yields
%
\begin{equation}\label{step1,2}
\phi(c_{u,v})=e^{-C/2}\sqrt{\frac{1}{2\pi v}}(1+o_{t}) .
\end{equation}

Now the proof follows easily by invoking the well-known approximation
to the tail probability of the standard normal distribution
%
\begin{equation}\label{step1,1}
P(|Z|>c)=\frac{2\phi(c)}{c}\bigl(1-z_1(c)\bigr) ,
\end{equation}
where
$z_1(c)$ is a positive function such that $z_1(c)c^2=O(1)$ as $c
\rightarrow\infty$.

The formula for type II error immediately follows from (\ref{type2,1})
and Assumption~\ref{Assumption(A)}.
\end{pf}

\subsection{The Bayes risk}\label{THB}

Under an additive loss function, the Bayes risk for a multiple testing
procedure is given by
%
\begin{equation}\label{risk01}
R = \delta_0 E(V) + \delta_A E(T) ,
\end{equation}
where $E(V)$ and $E(T)$ are the expected numbers of false positives and
false negatives, respectively.
In particular, under our mixture model, the Bayes risk for a fixed
threshold multiple testing procedure is given by
%
\begin{equation}\label{risk11}
R=m\bigl((1-p)t_1\delta_0+pt_2\delta_A\bigr) .
\end{equation}

Equations (\ref{t1,0}) and (\ref{t2n}) easily yield the following
asymptotic approximation to the optimal Bayes risk.
\begin{thr}\label{riskopt}
Under Assumption~\ref{Assumption(A)}, using the Bayes
oracle, the risk takes the form
%
\begin{equation}\label{optrisk1}
R_{\mathrm{opt}}=mp\delta_{A} \bigl(2\Phi\bigl(\sqrt{C}\bigr)-1\bigr)(1+o_{t}) .
\end{equation}
\end{thr}
\begin{rem}\label{assymetry}
It is important to note that under Assumption~\ref{Assumption(A)},
the asymptotic
form of the risk of the Bayes oracle in (\ref{riskopt}) is determined
by its type II error component. In fact the probability of type II
error, $t_2$, is much less sensitive to changes in the threshold value
than the probability of type I error, $t_1$. In particular, it is easy
to see that the same asymptotic form of $t_2$ [as in (\ref{t2n})] is
achieved by any multiple testing rule rejecting the null hypothesis
$H_{0i}$ when $X_i^2/\sigma^2 >c_t^2$, with $c_t^{2}=\log v +z_t$ and
$z_t=o(\log v)$. Probability of type I error is substantially more
sensitive to changes in the critical value, even if $z_t = o(\log v)$.
If $z_t$ is always positive, then the \textit{rate} of convergence of the
probability of type I error to zero is faster than that of the optimal
rule, and the total risk is still determined by the type II component.
Therefore the rule remains optimal as long as $z_t = o(\log v)$.
However, if $z_t=o(\log v)$ can take negative values, the situation is
quite subtle. In this case the rate of convergence of the probability
of type I error to zero may be equal or slower than that of the optimal
rule, making the overall risk of the rule substantially larger than
$R_{\mathrm{opt}}$. These observations are formally summarized in
Theorem~\ref{main}, which gives a characterization of the set of the asymptotically
optimal fixed threshold multiple testing rules.
\end{rem}
%
%
\begin{definition*}
Consider a sequence of parameter vectors
$\gamma_t$, satisfying Assumption~\ref{Assumption(A)}.
We call a multiple testing rule asymptotically Bayes optimal under
sparsity (ABOS) for $\gamma_t$ if its risk $R$ satisfies
\[
\frac{R}{R_{\mathrm{opt}}} \rightarrow1 \qquad\mbox{as } t\rightarrow\infty,
\]
where $R_{\mathrm{opt}}$ is the optimal risk, given by Theorem~\ref{riskopt}.
\end{definition*}
\begin{rem}
This definition relates optimality to a particular sequence of
$\gamma$ vectors satisfying Assumption~\ref{Assumption(A)}. However,
the asymptotically
optimal rule for a specific sequence $\gamma_t$ is also typically
optimal for a large set of ``similar'' sequences. The asymptotic
results presented in the following sections of this paper characterize
these ``domains'' of optimality for some of the popularly used multiple
testing rules. Since Assumption~\ref{Assumption(A)} is an inherent part
of our
definition of optimality, we will refrain from explicitly stating it
when reporting our asymptotic optimality results.
\end{rem}

The following theorem fully characterizes the set of asymptotically
Bayes-optimal multiple testing rules with fixed thresholds.
\begin{thr}\label{main}
A multiple testing rule of the form (\ref{BO}) with threshold
$c^2=c^2_t=\log v + z_{t}$ is ABOS if and only if
%
\begin{equation}\label{con1}
z_{t}=o(\log v)
\end{equation}
and
%
\begin{equation}\label{popr}
z_{t}+ 2\log\log v \rightarrow\infty.
\end{equation}
\end{thr}

The proof of Theorem~\ref{main} is provided in Section 8 of~\cite{App}.
\begin{rem}
Conditions (\ref{con1}) and (\ref{popr}) guarantee the
asymptotic Bayes optimality of the components of risk corresponding to
type II and type I errors, respectively.
\end{rem}

In the following corollary we present multiple testing rules which are
ABOS in the generic situation of Remark~\ref{Rem_gen}, where $u
\propto
-\log p$.\vadjust{\goodbreak}
\begin{cor}\label{lemRIC}
Assume (\ref{bound_logdelta}) holds, $\delta$ is bounded from above,
$m\rightarrow\infty$ and $p \propto m^{-\beta}$, with
$\beta>0$. Then a fixed threshold multiple testing rule (\ref{BO})
based on the threshold
%
\begin{equation}\label{RIC}
c^2=c^2_{m}= 2 \beta\log m+ d ,
\end{equation}
where $d\in\mathbb{R}$,
is ABOS.
\end{cor}

The proof is straightforward and is thus skipped.
\begin{rem}
The optimal threshold, provided in Corollary~\ref{lemRIC},
depends on the unknown parameter $\beta$.
It may be pointed out that it is proved in Section~\ref{Section:rules} that the
Benjamini--Hochberg multiple testing procedure adapts to this unknown
sparsity and, under very mild conditions on $\delta$ and the FDR level
$\alpha$, is ABOS whenever $0<\beta\leq1$. Corollary~\ref{lemRIC}
shows also that the universal threshold $2\log m$ of~\cite{DJ1} is ABOS
when $\beta=1$. Thus, within our asymptotic framework, the universal
threshold is asymptotically optimal when the expected number of true
signals does not increase with $m$.
\end{rem}

\section{Controlling the Bayesian False Discovery Rate} \label{Section:BFDR}

In a seminal paper~\cite{BH}, Benjamini and Hochberg introduced the
False Discovery Rate (FDR) as a measure of the accuracy of a multiple
testing procedure
%
\begin{equation}\label{FDR}
\mathrm{FDR}=E\biggl(\frac{V}{R}\biggr) .
\end{equation}
Here $R$ is the total number of null hypotheses rejected, $V$ is the
number of ``false'' rejections and it is assumed that $\frac{V}{R}=0$
when $R=0$. For tests with a fixed threshold, Efron and Tibshirani
\cite{ET} define another very similar measure, called the Bayesian False
Discovery Rate, BFDR,
%
\begin{equation}\label{BFDR}\quad
\mathrm{BFDR}=P(H_{0i} \mbox{ is true}| H_{0i} \mbox{ was
rejected})=\frac{(1-p)t_1}{(1-p)t_1+p(1-t_2)} ,
\end{equation}
where $t_1$ and $t_2$ are the probabilities of type I and type II
errors.

According to Theorem 1 of~\cite{S2}, in the case when individual test
statistics are generated by the
two-component mixture model and the multiple testing procedure uses
the same fixed threshold for each of the tests, $\mathrm{BFDR}$ coincides with
the positive False Discovery Rate $\mathrm{pFDR}$ of~\cite{S2}, defined as
\[
\mathrm{pFDR}=E\biggl(\frac{V}{R}\Big|R>0\biggr)=\frac{\mathrm{FDR}}{P(R>0)} .
\]

Note\vspace*{-1pt} here that in our context it is enough to consider threshold tests
that reject for high values of $\frac{X_i^2}{\sigma^2}$. This is\vspace*{1pt} due to
the fact that from the MLR property and the Neyman--Pearson lemma, it
can be easily proved that any other kind of test with the same type 1
error will have a larger BFDR and Bayesian False Negative Rate (BFNR).\vadjust{\goodbreak}

Extensive simulation studies and theoretical calculations in
\cite{GW1,BGT} and~\cite{BGOT} illustrate that multiple testing
rules controlling the BFDR at a small level $\alpha\approx0.05$
behave very well under sparsity in terms of minimizing the
misclassification error (i.e., the Bayes risk for $\delta_0=\delta_A$).
We also recall in this context that a test has BFDR $\alpha$ if and
only if
%
\begin{equation} \label{BFDRB}
(1-\alpha)(1-p)t_1+ \alpha p t_2 = \alpha p ,
\end{equation}
the left-hand side of (\ref{BFDRB}) being the Bayes risk for $\delta_0 =
1-\alpha$ and $\delta_{A} = \alpha$. So the definition of the BFDR
itself has a strong connection to the Bayes risk and a ``proper''
choice of $\alpha$ might actually yield an optimal rule (for similar
conclusions, see, e.g.,~\cite{SCai}). One can show quite easily that
under the mixture model (\ref{modelX}), the BFDR of a test based on the
threshold $c^2$ continuously decreases from $(1-p)$ for $c=0$ to $0$
for $c\rightarrow\infty$ (see Lemma 9.1 of~\cite{App}). In other
words, there exists a 1--1 mapping between thresholds $c\in[0,\infty)$
and BFDR levels $\alpha\in(0,1-p]$. So, if the BFDR control level is
chosen properly, the corresponding threshold can satisfy the conditions
of Theorem~\ref{main}.
\begin{rem}
In~\cite{Chi2} it is argued that when the data are
generated according to the two component mixture model, BFDR of any
fixed threshold rule as well as of the Benjamini--Hochberg procedure is
bounded from below by a constant $\beta^{\star}\geq0$, where $\beta
^{\star}$ depends on the actual mixture density. Lemma 9.1 of \cite
{App} shows under our mixture model (\ref{modelX}) $\beta^{\star}=0$,
that is, the criticality phenomenon of~\cite{Chi2} does not occur. This
is generally true in any case when the ratio of tail probabilities
$P(|X_i|>c)$ under the null and alternative distributions converges to
0 as $c\rightarrow\infty$.
\end{rem}

Now, we give a full characterization of asymptotically optimal BFDR
levels, which will be later used to prove ABOS of BH.

\subsection{ABOS of BFDR rules}

The general Theorem~\ref{Lu0}, below, gives conditions on $\alpha$,
which guarantee optimality for any given sequence of parameters $\gamma
_t$, satisfying Assumption~\ref{Assumption(A)}.
Corollary~\ref{Lu1} presents a special simple choice which works in the
general setting.
In the subsequent Corollary~\ref{Lu3}
we study the generic situation (\ref{bound_logdelta}) of Remark \ref
{Rem_gen}. Finally, Corollary~\ref{Lu4} considers the case where
$\alpha
=\mathrm{const}\in(0,1)$ and gives simple conditions for $\delta$ that guarantee
optimality.

Consider a fixed threshold rule (based on $\frac{X_i^2}{\sigma^2}$)
with the BFDR equal to~$\alpha$. Under the mixture model (\ref{modelX}),
a corresponding threshold value $c_{B}^2$ can be obtained by solving
the equation
%
\begin{equation}\label{BFDR1}
\frac{(1-p)(1-\Phi(c_B))}{(1-p)(1-\Phi(c_B))+p(1-\Phi
({c_B}/{\sqrt{u+1}}))}=\alpha,
\end{equation}
or equivalently, by solving
%
\begin{equation}\label{r1}
\frac{1-\Phi(c_B)}{1-\Phi({c_B}/{\sqrt{u+1}}
)}=\frac
{\alpha}{f(1-\alpha)}=\frac{r_{\alpha}}{f} ,
\end{equation}
where
%
\begin{equation}\label{ralpha}
r_{\alpha}=\frac{\alpha}{1-\alpha} .
\end{equation}

Note that $r_{\alpha}$ converges to 0 when $\alpha\rightarrow0$ and
to infinity when $\alpha\rightarrow1$.

Using Theorem~\ref{main}, one can show that this test is asymptotically
optimal only if $\frac{c_B}{\sqrt{u+1}}$ converges to $\sqrt{C}$, where
$C$ is the constant in Assumption~\ref{Assumption(A)}. From (\ref{r1}),
this in turn
implies that a BFDR rule for a chosen $\alpha$ sequence can only be
optimal if $\frac{r_\alpha}{f}$ goes to zero while satisfying certain
conditions. When $\frac{r_{\alpha}}{f} \rightarrow0$, a convenient
asymptotic expansion for $c_{B}^2$ can be obtained, and optimality
holds if and only if this asymptotic form conforms to the conditions
specified in Theorem~\ref{main}. The following theorem gives the
asymptotic expansion for $c^2_B$ and specifies the range of ``optimal''
choices of $r_{\alpha}$.
\begin{thr}\label{Lu0}
Consider a fixed threshold rule with $\mathrm{BFDR}=\alpha=\alpha_t$. Define
$s_t$ by
%
\begin{equation}\label{w0}
\frac{\log(f\delta\sqrt{u})}{\log(f/r_{\alpha})}=1+s_{t} ,
\end{equation}
where $r_{\alpha}=\frac{\alpha}{1-\alpha}$.
Then the rule is ABOS if and only if
%
\begin{equation}\label{w1}
s_{t}\rightarrow0
\end{equation}
and
%
\begin{equation}\label{w2}
2s_{t} \log(f/r_{\alpha})-\log\log(f/r_{\alpha})\rightarrow
-\infty.
\end{equation}

The threshold for this rule is of the form
%
\begin{equation}\label{asbfdr}
c^2_B=2\log\biggl(\frac{f}{r_{\alpha}}\biggr)-\log\biggl(2\log
\biggl(\frac{f}
{r_{\alpha}}\biggr)\biggr)+C_1+o_{t} ,
\end{equation}
where\vspace*{2pt}
$C_1=\log(\frac{2}{\pi D^2})$, and $D=2(1-\Phi(\sqrt{C}))$
is the asymptotic power. The corresponding probability of a type I
error is equal to
\[
t_1=D\frac{r_{\alpha}}{f}(1+o_{t}) .
\]
\end{thr}

The proof of Theorem~\ref{Lu0} can be found in Section 10 of~\cite{App}.
\begin{rem}
In comparison to (\ref{w1}), condition (\ref{w2})
imposes an additional restriction on positive values of $s_t$ (i.e.,
large values of $\alpha$). It is clear from the proof of Theorem \ref
{Lu0} that the necessity of this additional requirement results from
the asymmetric roles of type I and type II errors in the Bayes risk, as
discussed in Remark~\ref{assymetry}.
\end{rem}
\begin{rem}
Condition (\ref{w1}), given in Theorem~\ref{Lu0},
says (after some algebra) that a sequence of asymptotically optimal
BFDR levels $\alpha=\alpha_t$ satisfies $\frac{\alpha}{1-\alpha}=
(\delta\sqrt{u})^{b_t-1} f^{b_t}$ for some $b_t$, where $b_t
\rightarrow0$ as $t \rightarrow\infty$. Broadly speaking, this means
that for optimality the BFDR levels need to be chosen small when the
loss ratio is large. The seemingly evident dependence of $\alpha$ on
$u$ is not stressed in this article, since on the verge of
detectability $u=\frac{2}{C} \log(f \delta)(1+o_t)$ and, as seen in
the following corollaries, the range of asymptotically optimal levels
of $\alpha$ does not depend on $C$. A thorough discussion of the
dependence of $\alpha$ on $u$ in case when $C=0$ can be found in~\cite{BCFG}.
\end{rem}
\begin{cor}\label{Lu1}
A rule with BFDR at the level $\alpha=\alpha_t$, such that $r_{\alpha
}\propto(\delta\sqrt{u})^{-1}$, is ABOS.
Specifically, if $m\rightarrow\infty$, $p\propto m^{-\beta}$ ($\beta
>0$) and $\frac{\log\delta}{\log m} \rightarrow C_{\delta}\in
[0,\infty
]$, then a rule with BFDR at the level $\alpha$ such that $r_{\alpha}
\propto(\delta\sqrt{\log(m\delta)})^{-1}$ is ABOS.
\end{cor}
\begin{rem}\label{rem_ind}
Corollary~\ref{Lu1} shows that while the proposed optimal BFDR
level clearly depends on the ratio of losses $\delta$, it is
independent of the sparsity parameter $\beta$.
\end{rem}

The proof of Corollary~\ref{Lu1} is immediate by verifying that (\ref
{w1}) and (\ref{w2}) are satisfied by such sequences of $\alpha$'s. Also
the proofs of the following Corollaries~\ref{Lu3} and~\ref{Lu4},
follow quite immediately from Theorem~\ref{Lu0} and are thus omitted.
\begin{cor}\label{Lu3}
Assume the generic situation (\ref{bound_logdelta}) of Remark \ref
{Rem_gen}. Then a fixed threshold rule with BFDR equal to $\alpha$ is
ABOS if and only if $r_{\alpha}$ satisfies
\[
\log r_{\alpha}=o(\log p) \quad\mbox{and}\quad r_{\alpha}\delta\rightarrow0 .
\]
If we assume further that $m\rightarrow\infty$ and $p \propto
m^{-\beta
}$ ($\beta>0$), such a rule is ABOS if and only if
\[
\log r_{\alpha}=o(\log m) \quad\mbox{and}\quad r_{\alpha}\delta\rightarrow0 .
\]
In case when $\delta=\mathrm{const}$ and $p \propto m^{-\beta}$, the BFDR rule
is ABOS if and only if
$\alpha\rightarrow0$ such that $ \log\alpha=o(\log m)$.
\end{cor}
\begin{cor}\label{Lu4}
A fixed threshold rule with BFDR equal to $\alpha\in(0,1)$ is ABOS if
and only if $\delta\rightarrow0$ at such a rate that $\frac{\log
\delta
}{\log p} \rightarrow0$. If we assume that $m\rightarrow\infty$ and
$p \propto m^{-\beta}$ ($\beta>0$), such a rule is ABOS if and only if
$\delta\rightarrow0$ such that $\log\delta=o(\log m)$.
\end{cor}

Corollary~\ref{Lu4}, given above, states that a rule with BFDR at a
fixed level $\alpha$ is asymptotically optimal for a wide range of loss\vadjust{\goodbreak}
functions, such that $\delta\rightarrow0$. Note that the assumption
that $\delta\rightarrow0$ as $p\rightarrow0$ agrees with the
intuition that the cost of missing a signal should be relatively large
if the true number of signals is small. Corollary~\ref{Lu3} shows that
when the loss ratio is constant, a BFDR rule is asymptotically optimal
for a wide range of $\alpha$ levels, such that $\alpha\rightarrow0$.


\subsection{Optimality of the asymptotic approximation to the BH threshold}

In~\cite{GW1} it is proved that when the number of tests tends to
infinity, and the fraction of true alternatives remains fixed, then the
random threshold of the Benjamini--Hochberg procedure can be
approximated by
%
\begin{equation}\label{GW}
c_{\mathrm{GW}}\dvtx \frac{(1-\Phi(c_{\mathrm{GW}}))}{(1-p)(1-\Phi(c_{\mathrm{GW}}))+p(1-\Phi
({c_{\mathrm{GW}}}/{\sqrt{u+1}}))}=\alpha.
\end{equation}

Compared to the equation defining the BFDR rule (\ref{BFDR1}), the
function on the left-hand side of (\ref{GW}) lacks $(1-p)$ in the
numerator. In the case where $p\rightarrow0$ this term is negligible,
and one expects that the rule based on $c_{\mathrm{GW}}$ asymptotically
approximates the corresponding BFDR rule for the same $\alpha$. The
following result shows that this is indeed the case.
\begin{thr}\label{Lu6}
Consider the rule rejecting the null hypothesis $H_{0i}$ if $\frac
{X_i^2}{\sigma^2}\geq c^2_{\mathrm{GW}}$, where $c_{\mathrm{GW}}$ is defined in (\ref
{GW}). This rule is ABOS if and only if the corresponding BFDR rule
defined in (\ref{BFDR1})
is ABOS. In this case we have
\[
c^2_{\mathrm{GW}}=c^2_{B}+o_t ,
\]
where $c_{B}^2$ is the threshold of an asymptotically optimal BFDR
rule, defined in Theorem~\ref{Lu0}.
\end{thr}
\begin{pf}
Note that (\ref{GW}) is equivalent to
%
\begin{equation}\label{GWnew}
\frac{1-\Phi(c_{\mathrm{GW}})}{1-\Phi({c_{\mathrm{GW}}}/{\sqrt{u+1}}
)}=\frac{p r_{\alpha}}{1+p r_{\alpha}} = \frac{r_{\alpha^{\prime}}}{f},
\end{equation}
where $\alpha^{\prime}=\alpha(1-p)$. Thus $c_{\mathrm{GW}}$ is the same as the
threshold of a rule with BFDR at the level $\alpha^{\prime}$.

Define $s_{t^{\prime}}$ by $\frac{\log(f \delta\sqrt{u})}{\log
(f/r_{\alpha^{\prime}})} = 1+s_{t^{\prime}}$. It follows easily that
$s_{t^{\prime}}$ satisfies (\ref{w1}) and (\ref{w2}) of Theorem \ref
{Lu0} (with $\alpha$ replaced by $\alpha^{\prime}$), if and only if
$s_t$ defined in (\ref{w0}) satisfies (\ref{w1}) and (\ref{w2}). Thus
the first part of the theorem is proved.

To complete the proof of the theorem, we observe that the optimality of
a BFDR rule implies that $\frac{r_{\alpha}}{f} \rightarrow0$, and the
optimality of the rule based on $c_{\mathrm{GW}}$ implies that $\frac{r_{\alpha
^{\prime}}}{f} \rightarrow0$. In either case, $pr_{\alpha
}\rightarrow
0$ and thus (\ref{GWnew}) reduces to
%
\begin{equation}\label{cgw}
\frac{1-\Phi(c_{\mathrm{GW}})}{1-\Phi({c_{\mathrm{GW}}}/{\sqrt{u+1}})}=p
r_{\alpha}(1+o_t)=\frac{r_{\alpha}}{f}(1+o_t) .
\end{equation}
Now, the asymptotic approximation to $c^2_{\mathrm{GW}}$ can be obtained
analogously to the asymptotic form of the threshold for an optimal BFDR
rule, provided in (\ref{asbfdr}).
\end{pf}

\section{ABOS of classical frequentist multiple testing
procedures}\label{Section:rules}

Similarly to the Bayes oracle, the BFDR rules discussed in Section~\ref{Section:BFDR}
are not attainable, since they require the knowledge of the parameters
of the mixture distribution (\ref{modelX}). However, the results
included in Section~\ref{Section:BFDR} can be used to prove the asymptotic optimality of
classical multiple testing procedures, such as the Bonferroni rule and
the Ben\-ja\-mi\-ni--Hoch\-berg procedure (BH). In this section we
consider a sequence of problems in which the number of tests $m
\rightarrow\infty$ and the $\gamma$ sequence is indexed by $t=m$.

\subsection{ABOS of the Bonferroni correction}

The Bonferroni correction is one of the oldest and most popular
multiple testing rules. It is aimed at controlling the Family Wise
Error Rate, $\mathrm{FWER} =P(V>0)$, where $V$ is the number of false
discoveries. The Bonferroni correction at FWER level $\alpha$ rejects
all null hypothesis for which $Z_i=\frac{|X_i|}{\sigma}$ exceeds the threshold
\[
c_{\mathrm{Bon}}\dvtx 1-\Phi(c_{\mathrm{Bon}})=\frac{\alpha}{2m} .
\]

Under the assumption that $m\rightarrow\infty$, the threshold for the
Bonferroni correction can be written as
%
\begin{equation}\label{cbon}
c^2_{\mathrm{Bon}}=2\log\biggl(\frac{m}{\alpha}\biggr)-\log\biggl(2\log
\biggl(\frac
{m}{\alpha}\biggr)\biggr)+\log(2/ \pi)+o_m .
\end{equation}

Comparison of this threshold with the asymptotic approximation to an
optimal BFDR rule (\ref{asbfdr}) suggests that the Bonferroni
correction will have similar asymptotic optimality properties in the
extremely sparse case (\ref{sparse}). Indeed, these expectations are
confirmed by the following Lemma~\ref{Bon}, which will be used in the
next section for the proof of ABOS of the Benjamini--Hochberg procedure
under very sparse signals.
\begin{lem}\label{Bon}
Assume that $m \rightarrow\infty$ and (\ref{sparse}) holds.
The Bonferroni procedure at FWER level $\alpha_m \rightarrow\alpha
_{\infty}\in[0,1)$ is ABOS if $\alpha_m$ satisfies the assumptions of
Theorem~\ref{Lu0}.
\end{lem}
\begin{pf}
Under the assumptions of Lemma
\ref{Bon} and Theorem~\ref{Lu0}
\[
c^2_{\mathrm{Bon}}=c^2_B+2\log z_m -2 \log(1-\alpha_{\infty})+2 \log{D}+ o_m ,
\]
where $z_m=mp_m$, $D=2(1-\Phi(\sqrt{C}))$, and $c_{B}^2$ is the
threshold of the rule controlling the BFDR at level $\alpha_m$.
From (\ref{sparse}) it follows easily that $c^2_{\mathrm{Bon}} = c^2_{B}(1+o_m)$.
By assumption, the rule based on the threshold $c^2_{B}$ is optimal,
and hence $c^2_{\mathrm{Bon}}$ satisfies condition (\ref{con1}) of Theorem \ref
{main}. Condition (\ref{popr}) is satisfied, since by assumption $\log
z_m$ is bounded below for sufficiently large $m$ and thus ABOS of the
Bonferroni correction follows.
\end{pf}

\subsection{ABOS of BH}\label{optBH}

Let $Z_i=|\frac{X_i}{\sigma}|$ and $p_i=2(1-\Phi(Z_i))$ be
the corresponding $p$-value.
We sort $p$-values in ascending order $p_{(1)}\leq p_{(2)}\leq\cdots
\leq p_{(m)}$
and denote
%
\begin{equation}\label{BH}
k=\max\biggl\{i\dvtx p_{(i)}\leq\frac{i \alpha}{m}\biggr\}.
\end{equation}
The Benjamini--Hochberg procedure BH at FDR level $\alpha$ rejects all
the null hypotheses for which the corresponding $p$-values are
smaller than or equal to $p_{(k)}$.
\begin{rem}
BH gained large popularity after the seminal paper
\cite{BH}, where it was proved that it controls FDR. It was originally
proposed in~\cite{S}, and later used in~\cite{Sm} as a test for the
global null hypothesis.
\end{rem}

Let us denote $1-\hat F_m(y) = \# \{|Z_i|\geq y\}/m$. It is easy to
check (see, e.g., (2.2) of~\cite{Sen} or the equivalence theorem of
\cite{ET}) that the\vspace*{1pt} Benjamini--Hochberg procedure rejects the null
hypothesis $H_{0i}$ when $Z_i^2\geq\tilde c^2_{\mathrm{BH}}$, where
%
\begin{equation}\label{cBH}
\tilde c_{\mathrm{BH}}=\inf\biggl\{y\dvtx \frac{2(1-\Phi(y))}{1-\hat F_m(y)}\leq
\alpha\biggr\}.
\end{equation}

Note also that BH rejects the null hypothesis $H_{0i}$ whenever $Z_i^2$
exceeds the threshold of the Bonferroni correction. Therefore, we
define the random threshold for BH as
\[
c_{\mathrm{BH}}=\min\{c_{\mathrm{Bon}}, \tilde c_{\mathrm{BH}}\} .
\]

Comparing (\ref{cBH}) and (\ref{GW}), we observe that the difference
between $\tilde c_{\mathrm{BH}}$ and $c_{\mathrm{GW}}$ is in replacing the cumulative
distribution function of $|Z_i|$ [appearing in (\ref{GW})] by the
empirical distribution function (in~\ref{cBH}). Therefore, as shown in
\cite{GW1}, for any fixed mixture distribution (\ref{modelX}) $\tilde
c_{\mathrm{BH}}$ converges in probability to $c_{\mathrm{GW}}$ as $m\rightarrow\infty$.
The following Theorem~\ref{BHthreshold}, shows that the approximation
of $\tilde c_{\mathrm{BH}}$ by $c_{\mathrm{GW}}$ works also within our asymptotic
framework, where $p_m\rightarrow0$ and $c_{\mathrm{GW}}\rightarrow\infty$.
\begin{thr}\label{BHthreshold}
Assume that $p_m\rightarrow0$ such that for sufficiently large $m$
%
\begin{equation}\label{pm2}
p_m>\frac{log^{\beta_p}m}{m}\qquad \mbox{for some constant } \beta_p>1 .
\end{equation}
Moreover, assume that the sequence of FDR levels $\alpha_m$ satisfies
%
\begin{equation}\label{alphauf}
\alpha_m \rightarrow\alpha_{\infty}<1
\end{equation}
and
%
\begin{equation}
\label{alpha2}
\alpha_m \mbox{ satisfies the assumption of Theorem~\ref{Lu0}}.
\end{equation}
Then
for every $\varepsilon>0$, every constant $\beta_1 > 0$ and sufficiently
large $m$ (dependent on $\varepsilon$ and $\beta_1$)
\[
P(|c_{\mathrm{BH}}-c_{\mathrm{GW}}|>\varepsilon)\leq m^{-\beta_1}.\vadjust{\goodbreak}
\]
\end{thr}

The proof of Theorem~\ref{BHthreshold} is provided in Section 11 of
\cite{App}.

Theorem~\ref{BHthreshold} suggests asymptotic optimality of BH under a
relatively ``dense'' scenario, specified in assumption (\ref{pm2}).
Indeed, the following Theorem~\ref{BHthr}, shows asymptotic optimality
of BH and
extends the ``optimality'' range of the sparsity parameter to all
sequences $p_m$ such that $mp_m \rightarrow s \in(0,\infty]$.
Concerning type I error component of the risk, this extension was
possible due to the precise and powerful results of~\cite{FR} on the
expected number of false discoveries using BH under the total null
hypothesis. The optimality of the type II error component under the
extremely sparse scenario (\ref{sparse}) results directly from a
comparison with the Bonferroni correction and Lemma~\ref{Bon}.\vspace*{-2pt}
\begin{thr}\label{BHthr}
Assume that
%
\begin{equation}\label{pm}
m \rightarrow\infty,\qquad p_m\rightarrow0,\qquad m p_m\rightarrow s \in
(0,\infty] .
\end{equation}
Then BH at the FDR level $\alpha=\alpha_m$ is ABOS if (\ref{alphauf})
and (\ref{alpha2}) hold.\vspace*{-2pt}
\end{thr}

The proof of Theorem~\ref{BHthr} is provided in Section~\ref{App}.\vspace*{-2pt}
\begin{rem}
Theorem~\ref{BHthr} states that under the sparsity
assumption (\ref{pm}), BH behaves similarly to a BFDR control rule.
Specifically, if assumptions (\ref{alphauf}) and (\ref{pm}) are
satisfied, then the BH rule is ABOS under FDR-levels $\alpha\propto
(\delta\sqrt{u})^{-1}$, as in Corollary~\ref{Lu1}. Furthermore, if
$p\propto m^{-\beta}$, with $0<\beta\leq1$, $\frac{\log\delta
}{\log
m} \rightarrow C_{\delta}\in[0,\infty]$ and $\delta\sqrt{\log
(m\delta
)} \rightarrow\infty$, then a\vspace*{1pt} rule with FDR at the level $\alpha$ such
that $\alpha\propto(\delta\sqrt{\log(m\delta)})^{-1}$ is ABOS.
Also,\vspace*{1pt} in the case when $p\propto m^{-\beta}$ ($0<\beta\leq1$) and
$\delta\propto\frac{1}{\sqrt{\log m}}, $ then BH at a fixed FDR level
$\alpha\in(0,1)$ is ABOS. Thus, while the asymptotically optimal FDR
levels clearly depend on the ratio of losses $\delta$, they are
independent of the sparsity parameter $\beta$; that is, ABOS property
of BH is highly adaptive with respect to the level of sparsity.\vspace*{-2pt}
\end{rem}

The next Theorem~\ref{lemRIC1}, deals with optimality of BH under the
generic assumption (\ref{bound_logdelta}) which here has the form
$\log
\delta=o(\log m)$.\vspace*{-2pt}
\begin{thr}\label{lemRIC1}
Suppose $m\rightarrow\infty$ and $p \propto m^{-\beta}$, with
$0<\beta\leq1$. Moreover, assume that $\log\delta=o(\log m)$ and
$\alpha\rightarrow\alpha_{\infty}<1$. Then BH is ABOS if
\[
\log\alpha=o(\log m) \quad\mbox{and}\quad \alpha\delta\rightarrow0 .\vspace*{-2pt}
\]
\end{thr}
\begin{pf}
Given the assumptions we are in the situation of Corollary~\ref{Lu3},
and it is easy to verify that therefore all assumptions of
Theorem~\ref{BHthr} are fulfilled. Thus ABOS holds.\vspace*{-2pt}
\end{pf}
\begin{cor} Suppose $m\rightarrow\infty$ and $p \propto m^{-\beta}$, with
$0<\beta\leq1$. Moreover, assume that $\delta=\mathrm{const}$. Then BH is ABOS
if $\alpha$ converges to 0, such that $\log\alpha=o(\log m)$.\vadjust{\goodbreak}
\end{cor}
\begin{cor} Suppose $m\rightarrow\infty$ and $p \propto m^{-\beta}$, with
$0<\beta\leq1$. Moreover, assume that $\alpha=\mathrm{const}$. Then BH is ABOS
if $\delta$ converges to 0, such that $\log\delta=o(\log m)$.
\end{cor}


Theorem~\ref{BHthr}, Remark 5.2, Theorem~\ref{lemRIC1} and its
corollaries give some general suggestions on the choice of the optimal
FDR level for BH. Note, however, that according to Theorem~\ref{main},
BH can be asymptotically optimal even when the difference between its
asymptotic threshold $c_{\mathrm{GW}}$ and the threshold of the Bayes oracle
slowly diverges to infinity. The following lemma provides a more
specific condition on $\alpha$ and $\delta$, which guarantees that the
difference between $c_{\mathrm{GW}}$ and the threshold of the Bayes oracle
converges to a constant.

\begin{lem}\label{crit_diff}
Let $p_m \propto m^{-\beta}$, for some $\beta> 0$. Moreover, assume
that $\delta$ satisfies the generic assumption (\ref{bound_logdelta})
and that $\alpha$ satisfies the assumptions of Theorem~\ref{Lu0}. Then
the difference between the asymptotic approximation to the BH threshold
$c_{\mathrm{GW}}$ (\ref{GW}) and the threshold of the Bayes oracle (\ref{crit})
converges to a constant if and only if the FDR level $\alpha_m$ and the
ratio of loss functions $\delta_m$ satisfy the condition
%
\begin{equation}\label{concrit}
r_{\alpha_m} \delta_m=\frac{s_m}{\log m} ,
\end{equation}
where $r_{\alpha_m}=\frac{\alpha_m}{1-\alpha_m}$ and $s_m
\rightarrow
C_s \in(0,\infty)$.
\end{lem}
\begin{pf}
Straightforward algebra shows that the difference between the
threshold of the Bayes oracle and $c_{\mathrm{GW}}$ is equal to
\[
2 \log\log m + 2 \log(\delta_m r_{\alpha_m}) + \log(2 \beta/C) +
\log
(2 \beta) + C- C_1 + o_m ,
\]
where $C_1$ is the constant provided in (\ref{asbfdr}). From this Lemma
\ref{crit_diff} follows easily.
\end{pf}
\begin{rem}
Theorem~\ref{BHthreshold} states that if $\beta\in(0,1)$, then the
random threshold of BH can be well approximated by $c_{\mathrm{GW}}$. Therefore,
in this case Lemma~\ref{crit_diff} provides also the ``best''
asymptotically optimal choices of FDR levels for BH. Since under the
assumptions of Theorem~\ref{BHthreshold} $\alpha_m$ converges to a
constant smaller than one, condition (\ref{concrit}) can be written as
$\alpha_m \delta_m \propto(\log m)^{-1}$. Specifically, if $\delta
_m=\mathrm{const}$, then the sequence of best FDR levels should satisfy $\alpha
_m \propto(\log m)^{-1}$. Thus the choice $\alpha_m \propto(\log
m)^{-1}$ is recommended when one aims at minimizing the
misclassification rate. On the other hand, BH with the fixed FDR level
$\alpha\in(0,1)$ works particularly well if $\delta_m \propto(\log
m)^{-1}$.
\end{rem}

\section{Discussion}\label{Disc}

We have investigated the asymptotic optimality of multiple testing
rules under sparsity, using the framework of Bayesian decision theory.
We formulated conditions for the asymptotic optimality of the universal
threshold of~\cite{DJ1} and the Bonferroni correction. Moreover,
similarly to~\cite{A}, we have proved some asymptotic optimality\vadjust{\goodbreak}
properties of rules controlling the Bayesian False Discovery Rate and
the Benjamini and Hochberg procedure. Comparing with~\cite{A}, we
replaced a loss function based on estimation error with a loss function
dependent only on the type of testing error. This resulted in somewhat
different optimality properties of BH. Specifically, we have shown that
the optimal FDR level for BH depends on the ratio between the loss for
type I and type II errors and is almost independent of the level of
sparsity. Within our chosen asymptotic framework BH with the FDR levels
chosen in accordance with the assumed loss function is asymptotically
optimal in the entire range of sparsity parameters $p$, such that
$p\rightarrow0$ and $mp\rightarrow s \in(0, \infty]$. This range of
values of $p$ covers the situation when $p\propto1/m$, and in this way
it substantially extends the range of sparsity levels under which the
asymptotic minimax properties of BH were proved in~\cite{A}.

In this paper we proposed a new asymptotic framework to analyze
properties of multiple testing procedures. According to our definition
a multiple testing rule is ABOS if the ratio of its risk to the risk of
the Bayes oracle converges to 1 as the number of tests increases to
infinity. Our asymptotic results are to a large extend based on exact
inequalities for finite values of $m$. The refined versions of these
inequalities can be further used to characterize the rates of
convergence of the ratio of risks to 1 and to compare ``efficiency'' of
different ABOS methods. We consider this as an interesting area for
further research.

The results reported in this paper provide sufficient conditions for
the asymptotically optimal FDR levels for BH. They leave, however, a
lot of freedom in the choice of proportionality constants, which
obviously play a large role for a given finite value of $m$.
Based on the properties of BFDR controlling rules we expect that for
any given $m$ there exists FDR level $\alpha$ such that the risk of BH
is equal to the risk of the Bayes oracle. This finite sample optimal
choice of $\alpha$ would depend on the actual values of the mixture
parameters $p$ and $u$.
In recent years many Bayesian and empirical Bayes methods for multiple
testing have been proposed, which provide a natural way of
approximating the Bayes oracle in the case where the parameters of the
mixture distribution are unknown. The advantages of these Bayesian
methods, both in parametric and nonparametric settings, were
illustrated in, for example,~\cite{S3,E,BGOT,BGT,SB2}.
In~\cite{BGT} it is shown that when $p$ is \textit{moderately small} both
fully Bayesian and empirical Bayes methods perform very well with
respect to the Bayes risk. However, analysis of the asymptotic
properties of fully Bayesian methods in the case where $p_m\rightarrow
0$ remains a challenging task. In the case of empirical Bayes methods,
the asymptotic results given in~\cite{CJL}
illustrate that consistent estimation of the mixture parameters is
possible when $p_m \propto m^{-\beta}$, with $\beta\in(0,1)$. New
results on the convergence rates of these estimates, presented in~\cite
{CJ}, raise some hopes that proofs of the optimality properties of the
corresponding empirical Bayes rules can be found. It is, however,
rather unclear whether the full or empirical Bayes methods can be
asymptotically optimal in the extremely sparse case of $p_m \propto m^{-1}$.
Note that in this situation the expected number of signals does not
increase when $m\rightarrow\infty$ and consistent estimation of the
alternative distribution is not possible. These doubts, regarding the
asymptotic optimality of Bayesian procedures in the extremely sparse
case, are partially confirmed by the simulation study in~\cite{BGT},
where for very small $p$ Bayesian methods are outperformed by BH and
the Bonferroni correction at the traditional FDR and FWER levels
$\alpha=0.05$.

The Benjamini--Hochberg procedure can only be directly applied when the
distribution under the null hypothesis is completely
specified, that is, when $\sigma$ is known. In the case of testing a
simple null hypothesis (i.e., when $\sigma_0=0$), $\sigma$ can be
estimated using replicates. The precision of this estimation depends on
the number of replicates and can be arbitrarily good.
In the case where $\sigma_0>0$ (i.e., when we want to distinguish large
signals from background noise), the situation is quite
different. In this case, $\sigma$ can only be estimated by pooling the
information from all the test statistics. The related modifications of
the maximum likelihood method for estimating parameters in the sparse
mixture (\ref{modelX}) are discussed in~\cite{BGT}. More sophisticated
methods for estimating parameters of the normal null distribution in
case of no parametric assumptions on the form of the alternative are
provided in~\cite{E1} and~\cite{JinCai}. In~\cite{CJ} it is proved that
for $\beta<1/2$ the proposed estimators based on the empirical
characteristic function are minimax rate optimal. Simulation results
reported in~\cite{BGT} show that in the parametric setting of (\ref
{modelX}) and for very small $p$, the plug-in versions of BH at FDR
level $\alpha=0.05$ outperform Bayesian approximations to the oracle.
We believe that this is due to the fact that BH does not require the
estimation of $p$, which is rather difficult when $p$ is very small.
Despite this relatively good behavior of BH, it is rather unlikely that
the plug-in versions of BH are asymptotically optimal in the case where
$p\propto m^{-1}$. A thorough theoretical comparison of empirical Bayes
versions of BH with Bayesian approximations to the Bayes oracle and an
analysis of their asymptotic optimality remains an interesting problem
for future research.

Model (\ref{modelX}) assumes that the statistics for different tests
are independent. In principle, the model and the methods proposed in
this paper can be extended to cover the situation of dependent test
statistics. However, in that case the optimal Bayes solution for the
compound decision problem will be more difficult to obtain. In
particular the optimal Bayes classifier for the $i$th test may depend
on the values of all other test statistics, leading to a rather
complicated Bayes oracle. We believe that under specific dependency
structures BH may still retain its asymptotic optimality properties.
The detailed analysis of this problem requires a thorough new
investigation and remains an open problem for future research.

In this paper we have modeled the test statistics using a scale mixture
of normal distributions. As already mentioned, we believe that the main
conclusions of the paper will hold for a substantially larger family of
two component mixtures, which are currently often applied to multiple
testing problems (see, e.g., \mbox{\cite{E,E1,CJ}}).
In a recent article~\cite{CPS}, a new ``continuous'' one-group model
for multiple testing was proposed. As in our case, the test statistics
are assumed to have a normal distribution with mean equal to zero, but
the scale parameters are different for different tests and modeled as
independent random variables from the one-sided Cauchy distribution. As
discussed in~\cite{CPS}, the resulting Bayesian estimate of the vector
of means shrinks small effects strongly toward zero and leaves large
effects almost intact. In this way, it enables very good separation of
large signals from background noise. In~\cite{CPS} it is demonstrated
that the results from the proposed procedure for multiple testing often
agree with the results from Bayesian methods based on the two-group
model. A thorough analysis of the asymptotic properties of the method
proposed in~\cite{CPS} in the context of multiple testing remains a
challenging task. However, we believe that the suggested one-group
model has its own, very interesting virtues and Carvalho, Polson and
Scott~\cite{CPS} clearly demonstrate that the search for modeling
strategies for the problem of multiple testing, as well as for the most
meaningful optimality criteria, is still an open and active area of research.

\section{\texorpdfstring{Proof of Theorem \protect\ref{BHthr}}{Proof of Theorem 5.2}}\label{App}

The proof of Theorem~\ref{BHthr} consists of two parts. The first part
shows the optimality of the type I error component of the risk (see
Theorem~\ref{type1_bounb}) while the second part shows that of the type
II error component (see Theorem~\ref{BH_type2_dense}). Combining these
two facts, the result follows immediately. The proofs of Theorems \ref
{type1_bounb} and~\ref{BH_type2_dense} are based on a series of
intermediate results.

\subsection{Bound on the type \textup{I} error component of the risk}

The first and most essential step of the proof of the optimality of the
type I error component of the risk relies on showing that, under
certain conditions, the expected number of false discoveries of BH, $\mathit{E
V}$, is bounded by $c_v\alpha K$, where $\alpha$ is the FDR level, $K$
is the true number of signals and $c_v$ is a positive constant. This
result is very intuitive in view of the definition of FDR [see (\ref
{FDR})]. The proof is, however, nontrivial, due to the difference
between $E(\frac{V}{R})$ and $\frac{\mathit{E V}}{E R}$.
\begin{lem} \label{cond_bound}
Consider the BH rule at a fixed FDR level $\alpha\leq\alpha_0 < 1$.
Let $K$ be the number of true signals. The
conditional expected number of false rejections given that $K=k$, with
$k < m(\frac{1}{\alpha_0}-1)$, is bounded by
%
\begin{equation}\label{cond_exp_FP}
E(V|K=k) \leq\alpha\biggl( \frac{k}{1-\alpha} + \frac{1}{(1-\alpha
)^2}\biggr) .
\end{equation}
Specifically, for $1\leq k< m(\frac{1}{\alpha_0}-1) $
%
\begin{equation}\label{cond_exp_FP1}
E(V|K=k) \leq c_v \alpha k
\end{equation}
with
%
\begin{equation}\label{cvfinal}
c_v = \frac{2 - \alpha_0}{(1 - \alpha_0)^2} .
\end{equation}
\end{lem}
\begin{pf}
Given the condition $K=k$, there are $(m-k)$ true nulls. Let the
corresponding ordered $p$-values be $\tilde p_{(1)}\leq\cdots\leq\tilde
p_{(m-k)}$. Imagine that we apply to these $p$-values the following
procedure $\mathit{\tilde BH}_k$ which rejects the hypotheses whose $p$-values
are smaller than $ \tilde p_{(\tilde{k})}$, where
%
\begin{equation}\label{newBH}
\tilde{k}=\max\biggl\{i\dvtx \tilde p_{(i)}\leq\frac{\alpha(i+k)}{m}
\biggr\} .
\end{equation}
Let\vspace*{1pt} $\tilde{V}$ be the corresponding number of rejections. Then
$E(V|K=k) \leq E(\tilde{V})$, since the number of false rejections for
the original BH, $V$, is not larger than~$\tilde{V}$. Now, consider $m$
i.i.d. $p$-values $q_1,\ldots,q_m$ from the total null (i.e., each of
the $m$ nulls is true), which are independent of the given original
$p$-values. Let $\tilde q_{(1)} \leq\cdots\leq\tilde q_{(m-k)}$ be
the ordered values from the subsequence $q_1,\ldots,q_{m-k}$. Then
$\tilde q_{(1)}, \ldots, \tilde q_{(m-k)}$ and $\tilde p_{(1)}, \ldots
,\tilde p_{(m-k)}$ have exactly the same distribution. Let $V_1$ and
$V_2$ be the number of rejections of null when the procedure (\ref
{newBH}) is applied to the first $(m-k)$ or $m$ $q$'s, respectively.
Then $E(V|K=k) \leq E(\tilde{V}) = E(V_1) \leq E(V_2)$.

Now the bound on $k$ (see the assumption of Lemma~\ref{cond_bound})
guarantees that $\alpha(i + k) / m$ on the right-hand side of
(\ref
{newBH}) is smaller than 1 for all possible $i$. We can thus apply
Lemma 4.2 of~\cite{FR} directly, which yields
\[
E(V_2) = \alpha\sum_{i=0}^{m-1} (k+i+1) \pmatrix{m-1 \cr i} i!
\biggl(\frac
{\alpha}{m}\biggr)^i .
\]

Routine calculations now lead to Lemma~\ref{cond_bound}
\[
E(V_2) \leq\alpha\sum_{i=0}^{\infty} (k+i+1) \alpha^i =
\alpha\biggl( \frac{k}{1-\alpha} + \frac{1}{(1-\alpha)^2} \biggr) .
\]
\upqed\end{pf}
\begin{rem}
Note that in the case where $\alpha_0 < 0.5$, the inequality $k <
m(\frac{1}{\alpha_0}-1)$ is always fulfilled.
\end{rem}

The following lemma is an extension of Lemma~\ref{cond_bound} to the
mixture mod\-el~(\ref{modelX}).
\begin{lem} \label{type1_bound}
Under assumptions (\ref{alphauf}) and (\ref{pm}), the expected number
of false rejections is bounded by
\[
E(V)
< C_2 \alpha_m m p_m ,\vadjust{\goodbreak}
\]
where $C_2$ is any constant satisfying
\[
C_2>\cases{
\displaystyle \frac{2 - \alpha_\infty}{(1 - \alpha_\infty)^2},&\quad
when $s=\infty$,\vspace*{2pt}\cr
\displaystyle \frac{e^{-s}}{s (1-\alpha_\infty)^2}+ \frac{2 - \alpha_\infty}{(1 -
\alpha_\infty)^2}, &\quad when $s\in(0,\infty)$.}
\]
\end{lem}
\begin{pf}
Define $C_6 := \frac{1}{\alpha_\infty}-1$ and $m_0 := \min(m, C_6 m)$.
The following holds:
%
\begin{equation}\label{EV_bound}
E(V) \leq\sum_{k=0}^{m_0} E(V|K=k) P(K=k) + m P(K > m_0) .
\end{equation}
The first term can be bounded for $m$ large enough using Lemma \ref
{cond_bound},
\[
\sum_{k=0}^{m_0} E(V|K=k) P(K=k)
\leq\frac{\alpha_m}{(1-\alpha_m)^2} (1-p_m)^m + \tilde c_v \alpha
_m m
p_m ,
\]
where $\tilde c_v$ is any constant larger than $\frac{2-\alpha
_{\infty
}}{(1-\alpha_{\infty})^2}$.
Now observe that $\frac{1}{(1-\alpha_m)^2} (1-p_m)^m$ converges to 0 if
$s=\infty$ or to $\frac{e^{-s}}{(1-\alpha_\infty)^2}$ otherwise.
Hence, it follows that
\[
\sum_{k=0}^{m_0} E(V|K=k) P(K=k)
< C_2 m \alpha_m p_m ,
\]
for any constant $C_2$ satisfying the assumption of Lemma~\ref{type1_bound}.

Finally, note that the second term of (\ref{EV_bound}) vanishes for
$\alpha_\infty< 0.5$. On the other hand, for $\alpha_\infty\in
[0.5,1)$, Lemma 7.1 of~\cite{A} yields
\[
mP(K>m_0)=m P(K > C_6 m) \leq m \exp\bigl(- \tfrac{1}{4} mp_m h(C_6/p_m)\bigr) ,
\]
where $h(x) = \min(|x-1|, |x-1|^2)$. If $p_m \rightarrow0$, then for
any constant $C_7\in(0,C_6)$ and sufficiently large $m$,
the right-hand side is bounded from above by $ m \exp(- C_7 m
)\rightarrow0$.
Now, from the assumptions $mp_m \rightarrow s > 0$ and $\alpha
_m\rightarrow\alpha_{\infty}>0.5$, it follows that for any constant
$\beta_2>0$ and sufficiently large $m$, the second term of (\ref
{EV_bound}) is smaller than $\beta_2 \alpha_m m p_m$, and Lemma \ref
{type1_bound} follows.
\end{pf}

Lemma~\ref{type1_bound} easily leads to the following Theorem \ref
{type1_bounb}, on the optimality of the type I error component of the
risk of BH.
\begin{thr} \label{type1_bounb}
Under assumptions (\ref{alphauf})--(\ref{pm}), the type \textup{I} error
component of the risk of BH, $R_1=\delta_0 E(V)$, satisfies $\frac
{R_1}{R_{\mathrm{opt}}} \rightarrow0$, where $R_{\mathrm{opt}}$ is the optimal risk
defined in Theorem~\ref{riskopt}.
\end{thr}
\begin{pf}
From Lemma~\ref{type1_bound}
%
\begin{equation}\label{R1BH}
\frac{R_1}{R_{\mathrm{opt}}}=\frac{\delta_0 E(V)}{R_{\mathrm{opt}}}\leq C_3 \alpha_m
\delta_m(1+o_m) ,\vadjust{\goodbreak}
\end{equation}
where $C_3=\frac{C_2}{2\Phi(\sqrt{C})-1}$.
Now,\vspace*{1pt} observe that the left-hand side of (\ref{w2}) [included in assumption
(\ref
{alpha2})] can be written as
\[
2\log(\delta_m r_{\alpha_m})+\log u-\log\log(f/r_{\alpha_m}),
\]
and under (\ref{w1}) and Assumption~\ref{Assumption(A)} it can be
further reduced to
\[
2\log(\delta_m r_{\alpha_m})-\log C+o_m .
\]
Thus assumptions (\ref{w2}) and (\ref{alphauf}) together imply that
$\delta_m \alpha_m \rightarrow0$, and from (\ref{R1BH}) it immediately
follows that $\frac{R_1}{R_{\mathrm{opt}}}\rightarrow0$.
\end{pf}

\subsection{Bound on the type \textup{II} component of the risk}

To prove the optimality of the type II error component of the risk of
BH, we consider the extremely sparse case (\ref{sparse}) and the denser
case (\ref{dense}) separately. Note that in the extremely sparse case,
the optimality of the type II component of the risk of BH follows
directly from a comparison with the more conservative Bonferroni
correction, which according to Lemma~\ref{Bon} is ABOS in this range of
sparsity parameters.

The proof of optimality for the denser case is based on the
approximation of the random threshold of BH by the asymptotically
optimal threshold $c_{\mathrm{GW}}$ [see (\ref{GW})], given in Theorem \ref
{BHthreshold}. The corresponding ``denser'' case assumption (\ref
{pm2}) is substantially less restrictive than (\ref{dense}) and
partially covers the extremely sparse case (\ref{sparse}).
\begin{thr} \label{BH_type2_dense}
Under the assumptions of Theorem~\ref{BHthreshold} the type \textup{II} error
component of the risk of BH satisfies
%
\begin{equation}\label{t2risk}
R_2 \leq R_{\mathrm{opt}}(1+o_{m}) .
\end{equation}
\end{thr}
\begin{pf}
Denote the number of false negatives under the BH rule by $T$. Let us
fix $\varepsilon>0$ and let $\tilde c_1=c_{\mathrm{GW}}+\varepsilon$. Clearly,
\[
E(T)\leq E(T| c_{\mathrm{BH}}\leq\tilde c_1)P(c_{\mathrm{BH}}\leq\tilde c_1)+ m
P(c_{\mathrm{BH}} > \tilde c_1) ,
\]
and furthermore
\[
E(T| c_{\mathrm{BH}}\leq\tilde c_1)P(c_{\mathrm{BH}}\leq\tilde c_1)\leq E T_1 ,
\]
where $T_1$ is the number of false negatives produced by the rule based
on the threshold $\tilde c_1$.
Note that the rule based on $\tilde c_1$ differs from the
asymptotically optimal rule $c_{\mathrm{GW}}$ only by a constant, and therefore,
from Theorem~\ref{main}, it is asymptotically optimal. Hence, it
follows that $\delta_A E T_1=R_{\mathrm{opt}}(1+o_{m})$. On the other hand, from
Theorem~\ref{BHthreshold}, for any $\beta_1 > 0$ and sufficiently large
$m$ (dependent on $\varepsilon$ and $\beta_1$)
\[
P(c_{\mathrm{BH}} > \tilde c_1) \leq m^{-\beta_1} .
\]
Therefore,
\[
R_2=\delta_A E T \leq R_{\mathrm{opt}}(1+o_m)+\delta_A m^{1-\beta_1} .\vadjust{\goodbreak}
\]
Now, observe that under assumption (\ref{pm2})
\[
\frac{\delta_A m^{1-\beta_1}}{R_{\mathrm{opt}}}=C_4\frac{m^{-\beta
_1}}{p}<C_4\frac{m^{1-\beta_1}}{\log^{\beta_p}m} ,
\]
where $C_4=\frac{1}{2\Phi(\sqrt{C})-1}$. Thus, choosing, for example,
$\beta_1 = 1$, we conclude that $\delta_A m^{1-\beta_1}=o(R_{\mathrm{opt}})$,
and the proof is thus complete.
\end{pf}

\section*{Acknowledgments}
We want to thank two anonymous referees and the Associate Editor for
many constructive suggestions to improve this manuscript. We would also
like to express our gratitude to David Ramsey for careful reading of
the manuscript and helpful comments and to Krzysztof Bogdan for several
helpful suggestions.

\begin{supplement}
\stitle{Supplement to ``Asymptotic Bayes-optimality under sparsity of
some multiple testing procedures''}
\slink[doi]{10.1214/10-AOS869SUPP} 
\sdatatype{.pdf}
\sfilename{aos869\_suppl.pdf}
\sdescription{Analysis of behavior of BFDR for scale mixtures of normal distributions
and proofs of Theorems~\ref{main},~\ref{Lu0} and~\ref{BHthreshold}.}
\end{supplement}


%
\printaddresses


\begin{thebibliography}{50}

\bibitem{A}
\begin{barticle}[mr]
\bauthor{\bsnm{Abramovich},~\bfnm{Felix}\binits{F.}},
  \bauthor{\bsnm{Benjamini},~\bfnm{Yoav}\binits{Y.}},
  \bauthor{\bsnm{Donoho},~\bfnm{David~L.}\binits{D.~L.}} \AND
  \bauthor{\bsnm{Johnstone},~\bfnm{Iain~M.}\binits{I.~M.}}
(\byear{2006}).
\btitle{Adapting to unknown sparsity by controlling the false discovery rate}.
\bjournal{Ann. Statist.}
\bvolume{34}
\bpages{584--653}.
\bid{doi={10.1214/009053606000000074}, issn={0090-5364}, mr={2281879}}
\end{barticle}
\endbibitem

\bibitem{BH}
\begin{barticle}[mr]
\bauthor{\bsnm{Benjamini},~\bfnm{Yoav}\binits{Y.}} \AND
  \bauthor{\bsnm{Hochberg},~\bfnm{Yosef}\binits{Y.}}
(\byear{1995}).
\btitle{Controlling the false discovery rate: A~practical and powerful approach
  to multiple testing}.
\bjournal{J. Roy. Statist. Soc. Ser. B}
\bvolume{57}
\bpages{289--300}.
\bid{issn={0035-9246}, mr={1325392}}
\end{barticle}
\endbibitem

\bibitem{App}
\begin{bmisc}[auto:STB|2011-03-03|12:04:44]
\bauthor{\bsnm{Bogdan},~\bfnm{M.}\binits{M.}},
  \bauthor{\bsnm{Chakrabarti},~\bfnm{A.}\binits{A.}},
  \bauthor{\bsnm{Frommlet},~\bfnm{F.}\binits{F.}} \AND
  \bauthor{\bsnm{Ghosh},~\bfnm{J.~K.}\binits{J.~K.}}
(\byear{2010}).
\bhowpublished{Supplement to ``Asymptotic Bayes-optimality under sparsity of
  some multiple testing procedures.''
  \href{http://dx.doi.org/10.1214/10-AOS869SUPP}{DOI:10.1214/10-AOS869SUPP}.}
\end{bmisc}
\endbibitem

\bibitem{BCFG}
\begin{bmisc}[auto:STB|2011-03-03|12:04:44]
\bauthor{\bsnm{Bogdan},~\bfnm{M.}\binits{M.}},
  \bauthor{\bsnm{Chakrabarti},~\bfnm{A.}\binits{A.}},
  \bauthor{\bsnm{Frommlet},~\bfnm{F.}\binits{F.}} \AND
  \bauthor{\bsnm{Ghosh},~\bfnm{J.~K.}\binits{J.~K.}}
(\byear{2010}).
\bhowpublished{The Bayes oracle and asymptotic optimality of multiple testing
  procedures under sparsity. Available at
  \href{http://arxiv.org/abs/arXiv:1002.3501v1}{arXiv:1002.3501}}.
\end{bmisc}
\endbibitem

\bibitem{BGOT}
\begin{barticle}[auto:STB|2011-03-03|12:04:44]
\bauthor{\bsnm{Bogdan},~\bfnm{M.}\binits{M.}},
  \bauthor{\bsnm{Ghosh},~\bfnm{J.~K.}\binits{J.~K.}},
  \bauthor{\bsnm{Ochman},~\bfnm{A.}\binits{A.}} \AND
  \bauthor{\bsnm{Tokdar},~\bfnm{S.~T.}\binits{S.~T.}}
(\byear{2007}).
\btitle{On the Empirical Bayes approach to the problem of multiple testing}.
\bjournal{Quality and Reliability Engineering International}
\bvolume{23}
\bpages{727--739}.
\end{barticle}
\endbibitem

\bibitem{BGT}
\begin{bincollection}[mr]
\bauthor{\bsnm{Bogdan},~\bfnm{Ma{\l}gorzata}\binits{M.}},
  \bauthor{\bsnm{Ghosh},~\bfnm{Jayanta~K.}\binits{J.~K.}} \AND
  \bauthor{\bsnm{Tokdar},~\bfnm{Surya~T.}\binits{S.~T.}}
(\byear{2008}).
\btitle{A comparison of the {B}enjamini--{H}ochberg procedure with some
  {B}ayesian rules for multiple testing}.
In \bbooktitle{Beyond Parametrics in Interdisciplinary Research: {F}estschrift
  in Honor of {P}rofessor {P}ranab {K}. {S}en}.
\bseries{Inst. Math. Stat. Collect.}
\bvolume{1}
\bpages{211--230}.
\bpublisher{IMS}, \baddress{Beachwood, OH}.
\bid{doi={10.1214/193940307000000158}, mr={2462208}}
\end{bincollection}
\endbibitem

\bibitem{CJ}
\begin{barticle}[mr]
\bauthor{\bsnm{Cai},~\bfnm{T.~Tony}\binits{T.~T.}} \AND
  \bauthor{\bsnm{Jin},~\bfnm{Jiashun}\binits{J.}}
(\byear{2010}).
\btitle{Optimal rates of convergence for estimating the null density and
  proportion of nonnull effects in large-scale multiple testing}.
\bjournal{Ann. Statist.}
\bvolume{38}
\bpages{100--145}.
\bid{doi={10.1214/09-AOS696}, issn={0090-5364}, mr={2589318}}
\end{barticle}
\endbibitem

\bibitem{CJL}
\begin{barticle}[mr]
\bauthor{\bsnm{Cai},~\bfnm{T.~Tony}\binits{T.~T.}},
  \bauthor{\bsnm{Jin},~\bfnm{Jiashun}\binits{J.}} \AND
  \bauthor{\bsnm{Low},~\bfnm{Mark~G.}\binits{M.~G.}}
(\byear{2007}).
\btitle{Estimation and confidence sets for sparse normal mixtures}.
\bjournal{Ann. Statist.}
\bvolume{35}
\bpages{2421--2449}.
\bid{doi={10.1214/009053607000000334}, issn={0090-5364}, mr={2382653}}
\end{barticle}
\endbibitem

\bibitem{CPS}
\begin{bmisc}[auto:STB|2011-03-03|12:04:44]
\bauthor{\bsnm{Carvalho},~\bfnm{C.~M.}\binits{C.~M.}},
  \bauthor{\bsnm{Polson},~\bfnm{N.}\binits{N.}} \AND
  \bauthor{\bsnm{Scott},~\bfnm{J.~G.}\binits{J.~G.}}
(\byear{2008}).
\bhowpublished{The horseshoe estimator for sparse signals. Technical Report 2008-31,
Dept. Statistical Science, Duke~Univ.}
\end{bmisc}\vadjust{\goodbreak}
\endbibitem

\bibitem{Chi2}
\begin{barticle}[mr]
\bauthor{\bsnm{Chi},~\bfnm{Zhiyi}\binits{Z.}}
(\byear{2007}).
\btitle{On the performance of {FDR} control: Constraints and a partial
  solution}.
\bjournal{Ann. Statist.}
\bvolume{35}
\bpages{1409--1431}.
\bid{doi={10.1214/009053607000000037}, issn={0090-5364}, mr={2351091}}
\end{barticle}
\endbibitem

\bibitem{Chi}
\begin{barticle}[mr]
\bauthor{\bsnm{Chi},~\bfnm{Zhiyi}\binits{Z.}}
(\byear{2008}).
\btitle{False discovery rate control with multivariate {$p$}-values}.
\bjournal{Electron. J. Stat.}
\bvolume{2}
\bpages{368--411}.
\bid{doi={10.1214/07-EJS147}, issn={1935-7524}, mr={2411440}}
\end{barticle}
\endbibitem

\bibitem{DMT}
\begin{barticle}[mr]
\bauthor{\bsnm{Do},~\bfnm{Kim-Anh}\binits{K.-A.}},
  \bauthor{\bsnm{M{\"u}ller},~\bfnm{Peter}\binits{P.}} \AND
  \bauthor{\bsnm{Tang},~\bfnm{Feng}\binits{F.}}
(\byear{2005}).
\btitle{A {B}ayesian mixture model for differential gene expression}.
\bjournal{J. Roy. Statist. Soc. Ser. C}
\bvolume{54}
\bpages{627--644}.
\bid{doi={10.1111/j.1467-9876.2005.05593.x}, issn={0035-9254}, mr={2137258}}
\end{barticle}
\endbibitem

\bibitem{Djin}
\begin{barticle}[mr]
\bauthor{\bsnm{Donoho},~\bfnm{David}\binits{D.}} \AND
  \bauthor{\bsnm{Jin},~\bfnm{Jiashun}\binits{J.}}
(\byear{2004}).
\btitle{Higher criticism for detecting sparse heterogeneous mixtures}.
\bjournal{Ann. Statist.}
\bvolume{32}
\bpages{962--994}.
\bid{doi={10.1214/009053604000000265}, issn={0090-5364}, mr={2065195}}
\end{barticle}
\endbibitem

\bibitem{Djin2}
\begin{barticle}[mr]
\bauthor{\bsnm{Donoho},~\bfnm{David}\binits{D.}} \AND
  \bauthor{\bsnm{Jin},~\bfnm{Jiashun}\binits{J.}}
(\byear{2006}).
\btitle{Asymptotic minimaxity of false discovery rate thresholding for sparse
  exponential data}.
\bjournal{Ann. Statist.}
\bvolume{34}
\bpages{2980--3018}.
\bid{doi={10.1214/009053606000000920}, issn={0090-5364}, mr={2329475}}
\end{barticle}
\endbibitem

\bibitem{DJ1}
\begin{barticle}[mr]
\bauthor{\bsnm{Donoho},~\bfnm{David~L.}\binits{D.~L.}} \AND
  \bauthor{\bsnm{Johnstone},~\bfnm{Iain~M.}\binits{I.~M.}}
(\byear{1994}).
\btitle{Minimax risk over {$l\sb p$}-balls for {$l\sb q$}-error}.
\bjournal{Probab. Theory Related Fields}
\bvolume{99}
\bpages{277--303}.
\bid{doi={10.1007/BF01199026}, issn={0178-8051}, mr={1278886}}
\end{barticle}
\endbibitem


\bibitem{E1}
\begin{barticle}[mr]
\bauthor{\bsnm{Efron},~\bfnm{Bradley}\binits{B.}}
(\byear{2004}).
\btitle{Large-scale simultaneous hypothesis testing: The choice of a null
  hypothesis}.
\bjournal{J. Amer. Statist. Assoc.}
\bvolume{99}
\bpages{96--104}.
\bid{doi={10.1198/016214504000000089}, issn={0162-1459}, mr={2054289}}
\end{barticle}
\endbibitem

\bibitem{E}
\begin{barticle}[mr]
\bauthor{\bsnm{Efron},~\bfnm{Bradley}\binits{B.}}
(\byear{2008}).
\btitle{Microarrays, empirical {B}ayes and the two-groups model}.
\bjournal{Statist. Sci.}
\bvolume{23}
\bpages{1--22}.
\bid{doi={10.1214/07-STS236}, issn={0883-4237}, mr={2431866}}
\end{barticle}
\endbibitem

\bibitem{ET}
\begin{barticle}[pbm]
\bauthor{\bsnm{Efron},~\bfnm{Bradley}\binits{B.}} \AND
  \bauthor{\bsnm{Tibshirani},~\bfnm{Robert}\binits{R.}}
(\byear{2002}).
\btitle{Empirical Bayes methods and false discovery rates for microarrays}.
\bjournal{Genet. Epidemiol.}
\bvolume{23}
\bpages{70--86}.
\bid{doi={10.1002/gepi.1124}, issn={0741-0395}, pmid={12112249}}
\end{barticle}
\endbibitem

\bibitem{ETST}
\begin{barticle}[mr]
\bauthor{\bsnm{Efron},~\bfnm{Bradley}\binits{B.}},
  \bauthor{\bsnm{Tibshirani},~\bfnm{Robert}\binits{R.}},
  \bauthor{\bsnm{Storey},~\bfnm{John~D.}\binits{J.~D.}} \AND
  \bauthor{\bsnm{Tusher},~\bfnm{Virginia}\binits{V.}}
(\byear{2001}).
\btitle{Empirical {B}ayes analysis of a microarray experiment}.
\bjournal{J. Amer. Statist. Assoc.}
\bvolume{96}
\bpages{1151--1160}.
\bid{doi={10.1198/016214501753382129}, issn={0162-1459}, mr={1946571}}
\end{barticle}
\endbibitem

\bibitem{Finner}
\begin{barticle}[mr]
\bauthor{\bsnm{Finner},~\bfnm{Helmut}\binits{H.}},
  \bauthor{\bsnm{Dickhaus},~\bfnm{Thorsten}\binits{T.}} \AND
  \bauthor{\bsnm{Roters},~\bfnm{Markus}\binits{M.}}
(\byear{2009}).
\btitle{On the false discovery rate and an asymptotically optimal rejection
  curve}.
\bjournal{Ann. Statist.}
\bvolume{37}
\bpages{596--618}.
\bid{doi={10.1214/07-AOS569}, issn={0090-5364}, mr={2502644}}
\end{barticle}
\endbibitem

\bibitem{FR}
\begin{barticle}[mr]
\bauthor{\bsnm{Finner},~\bfnm{H.}\binits{H.}} \AND
  \bauthor{\bsnm{Roters},~\bfnm{M.}\binits{M.}}
(\byear{2002}).
\btitle{Multiple hypotheses testing and expected number of type {I} errors}.
\bjournal{Ann. Statist.}
\bvolume{30}
\bpages{220--238}.
\bid{doi={10.1214/aos/1015362191}, issn={0090-5364}, mr={1892662}}
\end{barticle}
\endbibitem

\bibitem{FBC}
\begin{bmisc}[auto:STB|2011-03-03|12:04:44]
\bauthor{\bsnm{Frommlet},~\bfnm{F.}\binits{F.}},
  \bauthor{\bsnm{Bogdan},~\bfnm{M.}\binits{M.}} \AND
  \bauthor{\bsnm{Chakrabarti},~\bfnm{A.}\binits{A.}}
(\byear{2010}).
\bhowpublished{Asymptotic Bayes optimality under sparsity for general priors
  under the alternative. Available at
  \href{http://arxiv.org/abs/arXiv:1005.4753v1}{arXiv:1005.4753v1}}.
\end{bmisc}
\endbibitem

\bibitem{GW1}
\begin{barticle}[mr]
\bauthor{\bsnm{Genovese},~\bfnm{Christopher}\binits{C.}} \AND
  \bauthor{\bsnm{Wasserman},~\bfnm{Larry}\binits{L.}}
(\byear{2002}).
\btitle{Operating characteristics and extensions of the false discovery rate
  procedure}.
\bjournal{J. R. Stat. Soc. Ser. B Stat. Methodol.}
\bvolume{64}
\bpages{499--517}.
\bid{doi={10.1111/1467-9868.00347}, issn={1369-7412}, mr={1924303}}
\end{barticle}
\endbibitem

\bibitem{GW2}
\begin{barticle}[mr]
\bauthor{\bsnm{Genovese},~\bfnm{Christopher}\binits{C.}} \AND
  \bauthor{\bsnm{Wasserman},~\bfnm{Larry}\binits{L.}}
(\byear{2004}).
\btitle{A stochastic process approach to false discovery control}.
\bjournal{Ann. Statist.}
\bvolume{32}
\bpages{1035--1061}.
\bid{doi={10.1214/009053604000000283}, issn={0090-5364}, mr={2065197}}
\end{barticle}
\endbibitem

\bibitem{GR}
\begin{barticle}[mr]
\bauthor{\bsnm{Guo},~\bfnm{Wenge}\binits{W.}} \AND
  \bauthor{\bsnm{Rao},~\bfnm{M.~Bhaskara}\binits{M.~B.}}
(\byear{2008}).
\btitle{On optimality of the {B}enjamini--{H}ochberg procedure for the false
  discovery rate}.
\bjournal{Statist. Probab. Lett.}
\bvolume{78}
\bpages{2024--2030}.
\bid{doi={10.1016/j.spl.2008.01.069}, issn={0167-7152}, mr={2528629}}
\end{barticle}
\endbibitem


\bibitem{JinCai}
\begin{barticle}[mr]
\bauthor{\bsnm{Jin},~\bfnm{Jiashun}\binits{J.}} \AND
  \bauthor{\bsnm{Cai},~\bfnm{T.~Tony}\binits{T.~T.}}
(\byear{2007}).
\btitle{Estimating the null and the proportional of nonnull effects in
  large-scale multiple comparisons}.
\bjournal{J. Amer. Statist. Assoc.}
\bvolume{102}
\bpages{495--506}.
\bid{doi={10.1198/016214507000000167}, issn={0162-1459}, mr={2325113}}
\end{barticle}
\endbibitem

\bibitem{Leh1}
\begin{barticle}[mr]
\bauthor{\bsnm{Lehmann},~\bfnm{E.~L.}\binits{E.~L.}}
(\byear{1957}).
\btitle{A theory of some multiple decision problems. {I}}.
\bjournal{Ann. Math. Statist.}
\bvolume{28}
\bpages{1--25}.
\bid{issn={0003-4851}, mr={0084952}}
\end{barticle}
\endbibitem

\bibitem{Leh2}
\begin{barticle}[mr]
\bauthor{\bsnm{Lehmann},~\bfnm{E.~L.}\binits{E.~L.}}
(\byear{1957}).
\btitle{A theory of some multiple decision problems. {II}}.
\bjournal{Ann. Math. Statist.}
\bvolume{28}
\bpages{547--572}.
\bid{issn={0003-4851}, mr={0096338}}
\end{barticle}
\endbibitem


\bibitem{Lehmann}
\begin{barticle}[mr]
\bauthor{\bsnm{Lehmann},~\bfnm{E.~L.}\binits{E.~L.}},
  \bauthor{\bsnm{Romano},~\bfnm{Joseph~P.}\binits{J.~P.}} \AND
  \bauthor{\bsnm{Shaffer},~\bfnm{Juliet~Popper}\binits{J.~P.}}
(\byear{2005}).
\btitle{On optimality of stepdown and stepup multiple test procedures}.
\bjournal{Ann. Statist.}
\bvolume{33}
\bpages{1084--1108}.
\bid{doi={10.1214/009053605000000066}, issn={0090-5364}, mr={2195629}}
\end{barticle}
\endbibitem

\bibitem{MR}
\begin{barticle}[mr]
\bauthor{\bsnm{Meinshausen},~\bfnm{Nicolai}\binits{N.}} \AND
  \bauthor{\bsnm{Rice},~\bfnm{John}\binits{J.}}
(\byear{2006}).
\btitle{Estimating the proportion of false null hypotheses among a large number
  of independently tested hypotheses}.
\bjournal{Ann. Statist.}
\bvolume{34}
\bpages{373--393}.
\bid{doi={10.1214/009053605000000741}, issn={0090-5364}, mr={2275246}}
\end{barticle}
\endbibitem

\bibitem{MPRR}
\begin{barticle}[mr]
\bauthor{\bsnm{M{\"u}ller},~\bfnm{Peter}\binits{P.}},
  \bauthor{\bsnm{Parmigiani},~\bfnm{Giovanni}\binits{G.}},
  \bauthor{\bsnm{Robert},~\bfnm{Christian}\binits{C.}} \AND
  \bauthor{\bsnm{Rousseau},~\bfnm{Judith}\binits{J.}}
(\byear{2004}).
\btitle{Optimal sample size for multiple testing: The case of gene expression
  microarrays}.
\bjournal{J. Amer. Statist. Assoc.}
\bvolume{99}
\bpages{990--1001}.
\bid{doi={10.1198/016214504000001646}, issn={0162-1459}, mr={2109489}}
\end{barticle}
\endbibitem

\bibitem{Pena}
\begin{barticle}[auto:STB|2011-03-03|12:04:44]
\bauthor{\bsnm{Pe{\~n}a},~\bfnm{E.~A.}\binits{E.~A.}},
  \bauthor{\bsnm{Habiger},~\bfnm{J.~D.}\binits{J.~D.}} \AND
  \bauthor{\bsnm{Wu},~\bfnm{W.}\binits{W.}}
(\byear{2011}).
\btitle{Power-enhanced multiple decision functions controlling
  family-wise error and false discovery rates}.
\bjournal{Ann. Statist.}
\bvolume{39}
\bpages{556--583}.
\end{barticle}
\endbibitem

\bibitem{Roquain}
\begin{barticle}[mr]
\bauthor{\bsnm{Roquain},~\bfnm{Etienne}\binits{E.}} \AND
  \bauthor{\bparticle{van~de} \bsnm{Wiel},~\bfnm{Mark~A.}\binits{M.~A.}}
(\byear{2009}).
\btitle{Optimal weighting for false discovery rate control}.
\bjournal{Electron. J. Stat.}
\bvolume{3}
\bpages{678--711}.
\bid{doi={10.1214/09-EJS430}, issn={1935-7524}, mr={2521216}}
\end{barticle}
\endbibitem

\bibitem{Sk3}
\begin{barticle}[mr]
\bauthor{\bsnm{Sarkar},~\bfnm{Sanat~K.}\binits{S.~K.}}
(\byear{2006}).
\btitle{False discovery and false nondiscovery rates in single-step multiple
  testing procedures}.
\bjournal{Ann. Statist.}
\bvolume{34}
\bpages{394--415}.
\bid{doi={10.1214/009053605000000778}, issn={0090-5364}, mr={2275247}}
\end{barticle}
\endbibitem

\bibitem{SB}
\begin{barticle}[mr]
\bauthor{\bsnm{Scott},~\bfnm{James~G.}\binits{J.~G.}} \AND
  \bauthor{\bsnm{Berger},~\bfnm{James~O.}\binits{J.~O.}}
(\byear{2006}).
\btitle{An exploration of aspects of {B}ayesian multiple testing}.
\bjournal{J. Statist. Plann. Inference}
\bvolume{136}
\bpages{2144--2162}.
\bid{doi={10.1016/j.jspi.2005.08.031}, issn={0378-3758}, mr={2235051}}
\end{barticle}
\endbibitem

\bibitem{SB2}
\begin{barticle}[auto:STB|2011-03-03|12:04:44]
\bauthor{\bsnm{Scott},~\bfnm{J.~G.}\binits{J.~G.}} \AND
  \bauthor{\bsnm{Berger},~\bfnm{J.~O.}\binits{J.~O.}}
(\byear{2010}).
\btitle{Bayes and empirical-Bayes multiplicity adjustment in the
  variable-selection problem}.
\bjournal{Ann. Statist.}
\bvolume{38}
\bpages{2587--2619}.
\end{barticle}
\endbibitem

\bibitem{S}
\begin{barticle}[auto:STB|2011-03-03|12:04:44]
\bauthor{\bsnm{Seeger},~\bfnm{P.}\binits{P.}}
(\byear{1968}).
\btitle{A note on a method for the analysis of significance en masse}.
\bjournal{Technometrics}
\bvolume{10}
\bpages{586--593}.
\end{barticle}
\endbibitem

\bibitem{Sen}
\begin{barticle}[mr]
\bauthor{\bsnm{Sen},~\bfnm{Pranab~K.}\binits{P.~K.}}
(\byear{1999}).
\btitle{Some remarks on {S}imes-type multiple tests of significance}.
\bjournal{J.~Statist. Plann. Inference}
\bvolume{82}
\bpages{139--145}.
\bid{doi={10.1016/S0378-3758(99)00037-3}, issn={0378-3758}, mr={1736438}}
\end{barticle}
\endbibitem


\bibitem{Sm}
\begin{barticle}[mr]
\bauthor{\bsnm{Simes},~\bfnm{R.~J.}\binits{R.~J.}}
(\byear{1986}).
\btitle{An improved {B}onferroni procedure for multiple tests of significance}.
\bjournal{Biometrika}
\bvolume{73}
\bpages{751--754}.
\bid{doi={10.1093/biomet/73.3.751}, issn={0006-3444}, mr={0897872}}
\end{barticle}
\endbibitem


\bibitem{S2}
\begin{barticle}[mr]
\bauthor{\bsnm{Storey},~\bfnm{John~D.}\binits{J.~D.}}
(\byear{2003}).
\btitle{The positive false discovery rate: A {B}ayesian interpretation and the
  {$q$}-value}.
\bjournal{Ann. Statist.}
\bvolume{31}
\bpages{2013--2035}.
\bid{doi={10.1214/aos/1074290335}, issn={0090-5364}, mr={2036398}}
\end{barticle}
\endbibitem

\bibitem{S3}
\begin{barticle}[mr]
\bauthor{\bsnm{Storey},~\bfnm{John~D.}\binits{J.~D.}}
(\byear{2007}).
\btitle{The optimal discovery procedure: A new approach to simultaneous
  significance testing}.
\bjournal{J. R. Stat. Soc. Ser. B Stat. Methodol.}
\bvolume{69}
\bpages{347--368}.
\bid{doi={10.1111/j.1467-9868.2007.005592.x}, issn={1369-7412}, mr={2323757}}
\end{barticle}
\endbibitem

\bibitem{STS}
\begin{barticle}[mr]
\bauthor{\bsnm{Storey},~\bfnm{John~D.}\binits{J.~D.}},
  \bauthor{\bsnm{Taylor},~\bfnm{Jonathan~E.}\binits{J.~E.}} \AND
  \bauthor{\bsnm{Siegmund},~\bfnm{David}\binits{D.}}
(\byear{2004}).
\btitle{Strong control, conservative point estimation and simultaneous
  conservative consistency of false discovery rates: A unified approach}.
\bjournal{J. R. Stat. Soc. Ser. B Stat. Methodol.}
\bvolume{66}
\bpages{187--205}.
\bid{doi={10.1111/j.1467-9868.2004.00439.x}, issn={1369-7412}, mr={2035766}}
\end{barticle}
\endbibitem

\bibitem{SCai}
\begin{barticle}[mr]
\bauthor{\bsnm{Sun},~\bfnm{Wenguang}\binits{W.}} \AND
  \bauthor{\bsnm{Cai},~\bfnm{T.~Tony}\binits{T.~T.}}
(\byear{2007}).
\btitle{Oracle and adaptive compound decision rules for false discovery rate
  control}.
\bjournal{J. Amer. Statist. Assoc.}
\bvolume{102}
\bpages{901--912}.
\bid{doi={10.1198/016214507000000545}, issn={0162-1459}, mr={2411657}}
\end{barticle}
\endbibitem

\bibitem{TS}
\begin{barticle}[mr]
\bauthor{\bsnm{Tai},~\bfnm{Yu~Chuan}\binits{Y.~C.}} \AND
  \bauthor{\bsnm{Speed},~\bfnm{Terence~P.}\binits{T.~P.}}
(\byear{2006}).
\btitle{A multivariate empirical {B}ayes statistic for replicated microarray
  time course data}.
\bjournal{Ann. Statist.}
\bvolume{34}
\bpages{2387--2412}.
\bid{doi={10.1214/009053606000000759}, issn={0090-5364}, mr={2291504}}
\end{barticle}
\endbibitem

\end{thebibliography}
\end{document}